\documentclass[11pt,reqno]{amsart}
\usepackage[letterpaper]{geometry}
\usepackage{graphicx}
\usepackage{bm}
\usepackage{amssymb,amsmath,amsthm,xcolor,mathrsfs}
\usepackage{stmaryrd}
\usepackage[final]{hyperref} 
\usepackage{enumitem}

\makeatletter\let\over\@@over\makeatother 

\newcommand{\I}{\texorpdfstring{\hyperref[sec non-singular]{I}}{I}}
\newcommand{\NS}{\texorpdfstring{\hyperref[sec non-singular]{I-A}}{I-A}}
\newcommand{\PV}{\texorpdfstring{\hyperref[sec point]{I-B}}{I-B}}
\newcommand{\VS}{\texorpdfstring{\hyperref[sec sheets]{II}}{II}}

\newcommand{\be}{\begin{equation} }
\newcommand{\ee}{\end{equation}}
\newcommand{\bse}{\begin{subequations}}
\newcommand{\ese}{\end{subequations}}

\newcommand{\jbracket}[1]{\langle{#1}\rangle}
\newcommand\F{\mathcal F}
\newcommand\bdd{{\mathrm{bdd}}}

\newcommand{\supp}[1]{\operatorname{supp}{#1}}

\newcommand{\diam}[1]{\operatorname{diam}{#1}}
\newcommand{\jump}[1]{\left\llbracket{#1}\right\rrbracket}

\newcommand{\maps}{\colon}

\newcommand{\grad}{\nabla}

\newcommand{\R}{\mathbb{R}}

\newcommand{\n}[2][]{#1\lVert #2 #1\rVert}
\newcommand{\abs}[2][]{#1\lvert #2 #1\rvert}

\theoremstyle{plain} 
\newtheorem{theorem}{Theorem}[section] 
 
\newtheorem{corollary}[theorem]{Corollary}

\newtheorem{lemma}[theorem]{Lemma}

\theoremstyle{remark}
\newtheorem{remark}[theorem]{Remark}

\numberwithin{equation}{section}

\title[Localized vorticity asymptotics]{Existence, nonexistence, and asymptotics of deep water solitary waves with localized vorticity}
\author[R. M. Chen]{Robin Ming Chen}
\address{Department of Mathematics, University of Pittsburgh, Pittsburgh, PA 15260} 
\email{mingchen@pitt.edu}  

\author[S. Walsh]{Samuel Walsh}
\address{Department of Mathematics, University of Missouri, Columbia, MO 65211} 
\email{walshsa@missouri.edu} 

\author[M. H. Wheeler]{Miles H. Wheeler}
\address{Courant Institute of Mathematical Sciences, New York University, New York, NY 10012}
\email{mwheeler@cims.nyu.edu}

\subjclass[2010]{35B40, 35R35, 76B15, 76B25, 76B45, 76B47}
\keywords{localized vorticity, deep water, solitary water waves}

\begin{document}

\begin{abstract}
  In this paper, we study solitary waves propagating along the surface of an infinitely deep body of water in two or three dimensions.  The waves are acted upon by gravity and capillary effects are allowed --- but not required --- on the interface. We assume that the vorticity is localized in the sense that it satisfies certain moment conditions, and we permit there to be finitely many point vortices in the bulk of the fluid in two dimensions. We also consider a two-fluid model with a vortex sheet. 

  Under mild decay assumptions, we obtain precise asymptotics for the velocity field and free surface, and relate this to global properties of the wave. For instance, 
  we rule out the existence of waves whose free surface elevations have a single sign and of vortex sheets with finite angular momentum. Building
  on the work of Shatah, Walsh, and Zeng \cite{shatah2013travelling}, we also prove the existence of families of two-dimensional capillary-gravity waves with compactly supported vorticity satisfying the above assumptions. For these waves, we further show that the free surface is positive in a neighborhood of infinity, and that the asymptotics at infinity are linked to the net vorticity.
\end{abstract}

\maketitle

\setcounter{tocdepth}{1}
\tableofcontents

\section{Introduction}

Consider a traveling wave moving through an infinitely deep body of water in dimension $n=2$ or $3$. Mathematically, we model this as a solution of the free boundary incompressible  Euler problem that evolves by translating with a constant wave velocity $c$.  Through a change of variables, all time-dependence in the system can then be removed, allowing us to say that the water occupies the domain 
\begin{equation}\label{Omega def}
  \Omega := \{ x \in \R^n : x_n < \eta(x^\prime) \},
\end{equation}
where we are writing $x = (x',x_n) \in \R^{n-1} \times \R$. Here $S := \partial \Omega$ is free surface at the interface between the air and water. The fluid region $\Omega$ is unbounded in the (vertical) $x_n$ direction, as opposed to the finite depth case where $\Omega$ is bounded below by some hyperplane $\{x_n =-d\}$. 

In the moving frame, the velocity field $u = u(x)$ satisfies the steady incompressible Euler equations 
\begin{subequations}\label{steady euler plus bc}
\begin{equation}\label{steady euler u} 
  (u-c) \cdot \nabla u + \nabla P + g e_n = 0, \quad \nabla \cdot u = 0 \qquad \textrm{in } \Omega, 
\end{equation}
where $g > 0$ is the gravitational constant of acceleration, $P = P(x)$ is the pressure, and $c = (c^\prime, 0)$ is the (horizontal) wave velocity.  For convenience, we normalize the density of the water to unity. On the free boundary,  we impose the kinematic and dynamic conditions: 
\begin{equation} \label{u boundary cond} 
  (u-c) \cdot N = 0, \quad P = \sigma \nabla \cdot N \qquad \textrm{on } S,
\end{equation}
where $\sigma\geq 0$ is the coefficient of surface tension and $N$ is the outward unit normal to $S$.  This model neglects the dynamics in the atmosphere, supposing that it is a region of constant pressure that is normalized to $0$.  
The dynamic condition therefore mandates that the pressure across the interface experiences a jump proportional to the mean curvature.  We treat both capillary-gravity waves (for which $\sigma > 0$) and gravity waves (for which $\sigma = 0$).  
\end{subequations}

\emph{Solitary waves} are localized traveling waves whose free surface profiles $\eta$ vanish at infinity. They are among the oldest and most well-studied examples of nonlinear wave phenomena in mathematics. This paper is focused on the following fundamental questions: Can we classify the regimes that admit solitary waves? When such waves do exist, what can be said about their asymptotics and decay rates?  Are there any natural constraints on their form?  

Russell famously reported having observed a solitary wave moving through the relatively shallow waters of the Glasgow--Edinburgh canal in 1844~\cite{russell1844report}.  In the intervening century and a half, a well-developed rigorous theory has been established for solitary waves in the finite-depth setting (see, for instance, \cite{miles1980solitary, groves2004steady,amick1982stokes,craig1988symmetry}). The study of capillary-gravity solitary waves in infinite depth began much later with the numerical work of Longuet-Higgins~\cite{longuet1988limiting,longuet1989capillary} and the rigorous construction of Iooss and Kirrman~\cite{iooss1996capillary}. A three-dimensional existence theory which extends to infinite depth has recently been developed by Buffoni, Groves, and Wahl\'en~\cite{buffoni2016variational}.
For infinite-depth gravity solitary waves, there are instead a number of \emph{nonexistence} results. Craig~\cite{craig2002nonexistence} showed that there are no two- or three-dimensional waves of pure elevation or depression in the sense that $\eta \geq 0$ or $\eta \leq 0$ implies that the wave is trivial.  In two dimensions, without any assumptions on the sign of $\eta$, Hur \cite{hur2012no} proved that waves with the algebraic decay $\eta = O(1/|x'|^{1+\varepsilon})$ as $|x^\prime| \to \infty$ are trivial. For both capillary-gravity and gravity waves in two dimensions, Sun~\cite{sun1997analytical} showed that the decay $\eta = O(1/|x'|^{1+\varepsilon})$ automatically improves to $\eta = O(1/|x'|^2)$, and in this case ruled out the existence of waves of pure elevation or depression. Wheeler~\cite{wheeler2016integral} obtained similar results in three dimensions, and also found leading-order expressions for the asymptotic form of the waves.

In this paper, we are interested in the role of \emph{vorticity} $\omega$, which is the curl of the fluid velocity field $u$.  As usual, for two-dimensional flows we understand this to mean the scalar vorticity
\[ 
  \omega := \partial_{x_1} u_2 - \partial_{x_2} u_1.
\]
All of the theory discussed above pertains exclusively to irrotational waves where $\omega \equiv 0$.   Incoming currents, the wind in the air, or temperature gradients in the water can all generate vorticity.  Much of the recent activity in water waves has been dedicated to proving the existence of rotational steady waves in various regimes.  We direct the reader to the survey by Strauss~\cite{strauss2010steady} and monograph by Constantin~\cite{constantin2011book} for an overview of these developments.  

We will focus on solitary waves whose vorticity is localized, either in the sense that it vanishes at infinity or else is confined to the interface.  The former describes the situation where there are vortical structures like eddies in the flow, resulting in a concentration of vorticity in the near field.  The latter is often called a vortex sheet. Both are physically significant middle-points between irrotational waves and waves with vorticity throughout the fluid.

Our main contribution is to determine the asymptotic form that the velocity field and free surface must take and to rule out solitary waves in a number of regimes.
and shed light on a number of qualitative properties.  In particular, our results apply to the two families of two-dimensional capillary-gravity waves with compactly supported vorticity constructed recently by Shatah, Walsh, and Zeng \cite{shatah2013travelling}. Among other things, we confirm that the free surface $\eta$ must take on both positive and negative values, and is positive in a neighborhood of infinity. Our arguments are quite general in that they are independent of the dimension, the presence of surface tension, the near-field shape of $S$, and the specific form of $\omega$.  We are therefore able to extend these results to a wide variety of physical settings and vorticity distributions; a major consequence is the nonexistence of 
waves of pure elevation and depression. More precisely, we consider $\omega$ falling into one of the three cases detailed below.

\subsection*{Case \NS: Non-singular localized vorticity.} \label{sec non-singular}
In dimension $n=2$ or $3$, we weak study solutions of \eqref{steady euler plus bc} whose vorticity $\omega$ satisfies
\begin{subequations} \label{localized vorticity assumptions} 
\begin{equation}\label{vorticity assumptions}
  \omega \in L^1(\Omega) \cap L^\infty(\Omega),
  \quad 
  |x|^k \omega \in L^1(\Omega)
  \quad
  \text{for some $k > n^2$}.
\end{equation}
In three dimensions we assume also that the vorticity is tangential to the free surface:
\begin{equation}
  \label{3d vanishing assumption}
  \omega \cdot N = 0 \qquad \textrm{on } S.
\end{equation}
For a two-dimensional flow embedded in three dimensions, the vector vorticity $(0,0,\omega)$  is automatically normal to $N = (N_1,N_2,0)$.
\end{subequations}
The condition \eqref{3d vanishing assumption} appears for instance in \cite[Chapter 3.7]{saffman1992book}. Note that $\omega$ is divergence free in the distributional sense, and thus its normal trace on $S$ is well-defined as an element of $H^{-1/2}(S)$.  

Importantly, this class of vorticity distributions includes waves with \emph{vortex patches}, where the support of $\omega$ is compact and positively separated from $S$.  There are a wealth of results on vorticity of this type in the absence of a free surface; see for instance~\cite[Section~8.3]{majda2002vorticity}.
As we will discuss further below, two-dimensional traveling water waves with a vortex patch were first constructed rigorously by Shatah, Walsh, and Zeng \cite{shatah2013travelling}; to the best of our knowledge, no rigorous existence results are currently available for $n=3$.  

\subsection*{Case \PV: Localized vorticity with point vortices.} \label{sec point}
In two dimensions, we allow for the presence of finitely many point vortices. Denoting their positions by
\begin{align*}
  \{\xi^1, \ldots, \xi^M\} =: \Xi \subset \Omega
\end{align*}
and strengths by $\varpi^i \in \R$, this means that
\begin{equation}
  \label{point vortex assumptions}
  \omega = \sum_{i=1}^M \varpi^i \delta_{\xi^i} + \omega_{\mathrm{ac}},
  \qquad 
\end{equation}
where $\delta_{\xi^i}$ is the Dirac measure with unit mass centered at $\xi^i$ and the function $\omega_{\mathrm{ac}}$ satisfies the localization assumptions \eqref{localized vorticity assumptions}. The corresponding velocity field $u$ solves the incompressible Euler equations \eqref{steady euler u} in a distributional sense on $\Omega \setminus \Xi$, as well as the boundary conditions \eqref{u boundary cond} on $S$.  We note that $u$ fails to be $L^2$ on the neighborhood of any point in $\Xi$. As is customarily done, we assume that each vortex is advected by the vector field found by taking the full velocity field and subtracting its own singular contribution.  For traveling waves, this results in the following condition linking the wave velocity to the flow:
\begin{equation}
  c = \left(u- \frac{1}{2\pi} \varpi^i \nabla^\perp  \log{|\cdot -\xi^i|}  \right)\Big|_{\xi^i}, \qquad \textrm{for } i = 1, \ldots, M. \label{u vortex advection} 
\end{equation}
One can arrive at \eqref{u vortex advection} by taking vortex patch solutions to the full Euler system then shrinking the diameter of the patch to $0$; see \cite[Theorems 4.1, 4.2]{marchioro1994book}. Note that \eqref{u vortex advection} severely constrains the possible arrangements of the point vortex centers.

Point vortices are widely used in applications as idealizations of highly concentrated regions of vorticity; see, for instance, \cite{saffman1992book}. The existence of traveling capillary-gravity waves with a single point vortex was proved by Shatah, Walsh, and Zeng in \cite{shatah2013travelling}.  Earlier works by Ter-Krikorov \cite{terkrkorov1958vortex} and Filippov \cite{filippov1960vortex,filippov1961motion} study the case of solitary gravity waves in finite-depth with a single vortex.   Recently, Varholm constructed finite-depth capillary-gravity waves with one or more point vortices \cite{varholm2016solitary}.

\subsection*{Case \VS: Vortex sheet.} \label{sec sheets}
Finally, we consider two-fluid models with a water region $\Omega_- := \{ x : x_n < \eta(x') \}$ as well as an air region $\Omega_+ := \{ x : x_n > \eta(x') \}$. We set $\Omega = \Omega_- \cup \Omega_+$ and $S = \partial \Omega_- = \partial \Omega_+$.
The two regions have (possibly different) constant densities $\rho_\pm > 0$. Letting $u_\pm := u|_{\Omega_\pm}$ and $P_\pm := P|_{\Omega_\pm}$, we require that
\begin{subequations} \label{vortex sheet equations} 
\begin{equation*}
  (u_\pm - c) \cdot \nabla u_\pm + \frac{1}{\rho_\pm} \nabla P_\pm + g e_n = 0, \quad \nabla \cdot u_\pm = 0, \quad \nabla \times u_\pm = 0 \qquad \textrm{in } \Omega_\pm,
  \label{vortex sheet steady euler} 
\end{equation*}
and, on the boundary,
\begin{equation*}
  N_\pm \cdot (u_\pm - c) = 0, \qquad \jump{P} = +\sigma \nabla \cdot N_+
  = -\sigma \nabla \cdot N_- \qquad \textrm{on } S. \label{vortex sheet boundary cond} 
\end{equation*}
\end{subequations}
Here $N_\pm$ denote the outward unit normals to $\Omega_\pm$ on $S$, and $\jump{\cdot} := (\cdot)_+ - (\cdot)_-$ is the jump of a quantity over $S$.  
Notice that the fluid velocity is irrotational in each region, but has 
a jump discontinuity in its tangential component over the interface. Thus the vorticity $\omega$ is a singular continuous measure supported on $S$.  

Vortex sheets have been studied extensively in both the applied and mathematical literature.  The existence of two-dimensional capillary-gravity solitary waves with a vortex sheet was proved by Amick \cite{amick1994internal} and Sun \cite{sun1997solitary}, though both considered the situation where the water region is bounded below by a rigid ocean bed.  In \cite{sun2001twofluid}, Sun constructed two-dimensional periodic capillary-gravity waves where both layers are infinite.  

\subsection{Notation} \label{notation section}

Given a point $x \in \R^2$, we denote $x^\perp := (-x_2, x_1)$.  Similarly, the perpendicular gradient $\nabla^\perp := (-\partial_{x_2}, \partial_{x_1})$.  
We also use the Japanese bracket notation $\jbracket{x} := \sqrt{1+|x|^2}$ for $x \in \R^n$. Finally, we let $\gamma_n := 2\pi^{n/2} / \Gamma(n/2)$ denote the surface area of an $n$-dimensional unit ball. In particular $\gamma_2 = 2\pi$ and $\gamma_3 = 4\pi$. 

Given an open set $U \subset \R^n$, $k \in \mathbb{N}$, $\alpha \in (0,1)$, a weight $w \in C^0(\overline{U}; \R_+)$, and a function $f \in C^k(U;\R)$, we define the weighted H\"older norm
\begin{align*}
  \n f_{C_w^{k+\alpha}(U)}
  := \sum_{\abs \beta \le k} \n{w\partial ^\beta f}_{C^0(U)}
  + \sum_{\abs \beta = k} \n{w [\partial^\beta f]_{\alpha} }_{C^0(U)},
\end{align*}
where here $[f]_{\alpha}(x)$ is the local H\"older seminorm
\begin{align*}
  [f]_{\alpha}(x) := \sup_{\substack{\abs y < 1\\ x+y\in U}}
  \frac{\abs{f(x+y)-f(x)}}{\abs y^\alpha}.
\end{align*}
We denote by 
\[ C_w^{k+\alpha}(\overline{U}) := \left\{ f \in C^{k+\alpha}(\overline{U}) : \| f \|_{C_w^{k+\alpha}(U)} < \infty \right\}.\] 
Occasionally, we will also work with the space $C_{\bdd}^{k+\alpha}(\overline{U})$, which is defined as $C_w^{k+\alpha}(\overline{U})$ with $w \equiv 1$.

\subsection{Statement of results} \label{results section}
Our first theorem gives finer decay and asymptotic properties for the families of small-amplitude two-dimensional capillary-gravity waves with a point vortex and vortex patch constructed in~\cite{shatah2013travelling}.

\begin{theorem}[Existence] \label{existence theorem}  
  Let $w(x) := \jbracket{x}^2/\jbracket{x_2}$ and fix $\sigma > 0$ and $\alpha \in (0,1)$. Then there are $\varpi_0,\rho_0,\tau_0 > 0$ such that:
  \begin{enumerate}[label=\rm(\alph*)]
  \item There exists a family of two-dimensional capillary-gravity water waves 
  \[ 
  \mathscr{C}_{\mathrm{loc}}= \{ (\eta,u,c)(\varpi) : | \varpi | < \varpi_0 \} 
  \]
  with vorticity $\omega(\varpi) = \varpi \delta_{(0,-1)}$ bifurcating from the trivial state $(\eta,u,c)(0) = (0,0,0)$ and having the regularity
  \[ (\eta, u)(\varpi) \in C_w^{3+\alpha}(\R) \times C_w^{2+\alpha}\big(\overline{\Omega(\varpi)}\setminus \{(0,-1)\}\big).\]

\item There exists a family of two-dimensional capillary-gravity water waves
  \[ 
  \mathscr{S}_{\mathrm{loc}} = \left\{ (\eta, u,c)(\varpi,\rho,\tau) : |\varpi| < \varpi_0, ~ 0 < \rho <  \rho_0, ~ |\tau| < \tau_0 \right\},
  \]
  bifurcating from the trivial state
   with 
  $\supp{\omega(\varpi,\rho,\tau)} =: D(\varpi,\rho,\tau) \subset\subset \Omega(\varpi,\rho,\tau)$. Each of these waves lies in the space 
  \[ 
  (\eta,u)(\varpi,\rho,\tau)  \in C_w^{3+\alpha}(\R) \times C_w^{2+\alpha}(\overline{\Omega(\varpi,\rho,\tau)} \setminus D(\varpi,\rho,\tau)).
  \]

\item For both families $\mathscr{C}_{\mathrm{loc}}$ and $\mathscr{S}_{\mathrm{loc}}$, the free surface profile and velocity have the asymptotic form:
  \begin{equation}
    \begin{alignedat} {2}
      \eta & = 
      \frac{1}{2g} \frac{\varpi^2}{\gamma_2}
      \left( 1 +  O(\varpi^2) \right)  \frac{1}{x_1^2} 
      + O\left( \frac 1{\abs{x_1}^{2+\varepsilon}} \right),
      &\qquad& \textup{as } |x_1| \to \infty, \\
      u & = \frac{2\varpi}{\gamma_2}
      \nabla \left( \left( e_1 + O(\varpi) \right) \cdot \frac x{\abs x^2} \right)
       + O\left(\frac 1{|x|^{2+\varepsilon}} \right), 
      && \textup{as } |x| \to \infty, 
    \end{alignedat} \label{SWZ asymptotics} 
  \end{equation}
  for any $\varepsilon \in (0,1/3)$.  In particular, $\eta > 0$ in a neighborhood of infinity.    
  \end{enumerate}
\end{theorem}

Parts (a) and (b) build upon \cite[Theorem 2.1, Theorem 2.3]{shatah2013travelling}, where an existence theory is carried out in Sobolev spaces $H^k$ with $k$ arbitrarily large.  The main improvement is the use of weighted H\"older spaces, which give considerably more information about asymptotic behavior.  The proof is an application of the implicit function theorem. It relies on weighted estimates 
for Poisson equations on unbounded domains,  which have played an important role in previous works on the decay properties of water waves by Craig and Sternberg \cite{craig1988symmetry}, Amick \cite{amick1994internal}, Sun \cite{sun1996asymptotic,sun1997analytical}, and Hur \cite{hur2008symmetry}.  In part (c) we go further and extract the leading-order form of $\eta$ and $u$ at infinity.  This is a consequence of Theorems \ref{decay theorem} and Theorem \ref{integral identity theorem} presented below.  See Section~\ref{existence section} for further properties of $\mathscr{C}_{\mathrm{loc}}$ and $\mathscr{S}_{\mathrm{loc}}$.

As mentioned above, the techniques that we use to study the asymptotic properties of localized vorticity waves extend to a much more general setting: vorticities in Case \NS, \PV, and \VS, dimensions $n=2$ or $3$, gravity or capillary-gravity waves, and large and small amplitude. The main new assumption that we require is that $u$ (or $u_\pm$ in the case of vortex sheet) and $\eta$ exhibit the  decay:
\begin{subequations} \label{decay assumptions}
\begin{equation}
  \eta \in C^2_\bdd(\R^{n-1}),
  \quad 
  \partial^\beta \eta = O\left(\frac{1}{|x^\prime|^{ n-1+|\beta|+4\varepsilon}}\right) 
  \ 
  \textrm{as $|x^\prime| \to \infty$, $0 \leq |\beta| \leq 1$}  \label{eta decay} 
\end{equation}
and
\begin{equation}
  u = O\left( \frac{1}{|x|^{n-1+\varepsilon}}\right) \qquad \textrm{as } |x| \to \infty, \label{u decay} 
\end{equation}
\end{subequations}
for some $\varepsilon \in (0,1/4)$.  In fact, \eqref{u decay} can be replaced by the weaker condition $\varphi = o(1/|x|^{n-2})$, where $\varphi$ is the velocity potential for the irrotational part of the flow introduced in Section~\ref{splitting section}. We make the non-optimal assumption on the velocity field as it is a more physical quantity;  see Remark \ref{decay of phi remark}.

Our second result states that there are no waves of pure depression or pure elevation satisfying the above hypotheses.  
\begin{theorem}[Nonexistence] \label{no excess mass theorem}  
  Consider a solitary wave with localized vorticity 
in Case~\I\ or Case~\VS\ and suppose that and $u$ and $\eta$ have the decay \eqref{decay assumptions}. 
  If $\eta \geq 0$ or $\eta \leq 0$, then $\eta \equiv 0$.  Moreover, there is no excess mass:
  \[ 
  \int_{\R^{n-1}} \eta \, dx^\prime = 0.  
  \]
\end{theorem}
\begin{remark} \label{L1 remark}
  The decay assumption on $\eta$ in \eqref{eta decay} implies (just barely) that $\eta \in L^1$.
\end{remark}
\begin{remark} \label{trivial remark}
  Note that in Theorem~\ref{no excess mass theorem} we make no claim that $u \equiv 0$. Indeed, Constantin \cite{constantin2011dynamical} has constructed two-dimensional stationary waves ($c=0$) where $\eta \equiv 0$ and $u \equiv 0$ outside a perfectly circular region where the vorticity is non-constant. One can show that, in two dimensions, waves with $\eta \equiv 0$ and localized vorticity
  are necessarily stationary.
\end{remark}

For two-dimensional irrotational solitary waves and vortex sheets in infinite depth, this was proved by Sun~\cite{sun1997analytical}, and in the three-dimensional irrotational case it appears in \cite{wheeler2016integral}.  With strong surface tension ($\sigma > \abs c^2/4g$), Sun is able to assume even weaker decay than \eqref{decay assumptions}.
In the two- and three-dimensional irrotational cases, the nonexistence of pure waves of elevation or depression is a celebrated result due to Craig~\cite{craig2002nonexistence}, who obtained it using a maximum principle argument that avoids making decay assumptions as in \eqref{decay assumptions} but does not yield the stronger fact that $\int \eta \, dx^\prime = 0$.   More recently, Hur \cite{hur2012no} proved that there are no nontrivial two-dimensional gravity waves of any kind having $\eta = O(1/|x^\prime|^{1+\varepsilon})$.  To the best of our knowledge, the only prior nonexistence results for rotational water waves are due to Wahl\'en \cite{wahlen2014non}, who completely ruled out the possibility of three-dimensional solitary waves with constant vorticity in finite depth.
 
In our next theorem, we show that any solitary wave exhibiting the localization~\eqref{decay assumptions} necessarily decays even faster and has a specific asymptotic form.
Indeed, we find that $u$ must tend to a dipole velocity field, where the dipole moment is purely horizontal. 

\begin{theorem}[Asymptotic form] \label{decay theorem} Consider a solitary wave as in Theorem~\ref{no excess mass theorem}.
  \begin{enumerate}[label=\rm(\alph*)]
  \item  For vorticity of Case~\I,
    there exists a dipole moment $p = (p^\prime,0) \in \R^n$ such that
  \begin{subequations}\label{n-d better decay}
    \begin{equation}
      \eta = \frac{1}{g} p^\prime \cdot \nabla \left( \frac{c^\prime \cdot x^\prime}{|x^\prime|^n}  \right)  + O\left( \frac{1}{|x^\prime|^{n+\varepsilon}} \right), \qquad \textrm{as }  |x^\prime| \to \infty,
      \label{eta n-d better decay} 
    \end{equation}
    and
    \begin{equation}
      u = \nabla \left(   \frac{p \cdot x}{|x|^n} \right)  + O\left( \frac{1}{|x|^{n+\varepsilon}} \right), \qquad \textrm{as } |x| \to \infty. \label{u n-d better decay} 
    \end{equation}
  \end{subequations}

\item For vorticity of Case~\VS, there exists $p_\pm = (p_\pm^\prime,0) \in \R^n$ such that 
  \begin{subequations}\label{vortex sheet better decay}
    \begin{equation}
      \eta = \frac{1}{g\jump{\rho}} \jump{\rho p^\prime} \cdot \nabla \left( \frac{c^\prime \cdot x^\prime}{|x^\prime|^n} \right) + O\left( \frac{1}{|x^\prime|^{n+\varepsilon}} \right), \qquad \textrm{as } |x^\prime| \to \infty, \label{vortex sheet eta better decay} 
    \end{equation}
    and 
    \begin{equation}
      u_\pm = \nabla \left( \frac{p_\pm \cdot x}{|x|^n} \right) + O\left( \frac{1}{|x|^{n+\varepsilon}} \right), \qquad \textrm{as } |x| \to \infty. \label{vortex sheet u better decay} 
    \end{equation}
  \end{subequations}
  \end{enumerate}
\end{theorem}

\begin{remark} \label{2-d special case remark}
  In two dimensions, \eqref{eta n-d better decay} simplifies to
  \[ 
    \eta = -\frac{1}{g} \frac{p^\prime c^\prime}{|x^\prime|^2}  + O \left( \frac{1}{|x^\prime|^{2+\varepsilon}} \right), \qquad \textrm{as } |x^\prime| \to \infty. 
  \]
  Thus $\eta$ is strictly positive in a neighborhood of infinity whenever $p'c' < 0$.
\end{remark}

This generalizes the recent work of Wheeler \cite{wheeler2016integral} on the irrotational case. In the two-dimensional irrotational and vortex sheet cases, Sun \cite{sun1997analytical} established similar decay rates (but not asymptotics) under analogous assumptions. Such asymptotic behavior is assumed by Longuet-Higgins in the mostly numerical paper \cite{longuet1989capillary}, and is also considered by Benjamin and Olver \cite[Section~6.5]{benjamin1982hamiltonian}.

An important quantity describing the vorticity distribution in a water wave is the so-called \emph{vortex impulse} defined by 
\begin{equation}
\label{vortex impulse}
  m := \left\{\begin{array}{ll} 
    - \displaystyle \int_\Omega \omega (x - \xi^*)^{\perp} \, dx \quad & \text{if } n = 2, \\ \\
    - \displaystyle \frac 12 \int_\Omega \omega \times x \, dx  & \text{if } n = 3.
  \end{array}\right.
\end{equation}
See also \cite[Chapter 3.2, 3.7]{saffman1992book}. The following identity reveals a link between the vortex impulse $m$ and the dipole moment $p$ appearing in Theorem~\ref{decay theorem}. It is, in particular, essential to the proof of Theorem~\ref{existence theorem}(c).
\begin{theorem}[Dipole moment formula] \label{integral identity theorem} 
  Consider a solitary wave as in Theorem~\ref{no excess mass theorem}.
  \begin{enumerate}[label=\rm(\alph*)]
  \item  For vorticity of Case~\I,
    \begin{equation}
      \int_{\Omega} \left[  |u|^2 - (u-c) \cdot \left( V +  {1 \over \gamma_n} \nabla\left( {m \cdot x \over |x|^n} \right) \right) \right] \, dx = -\frac{\gamma_n}{2} \; c \cdot  p, \label{dipole formula} 
    \end{equation}
    where $V$ is the vortical part of the velocity defined in \eqref{n-d V def}.

\item  For vorticity of Case~\VS,
  \begin{equation}
    \int_\Omega \rho |u|^2  \, dx =  \frac{\gamma_n}{2}    c \cdot \jump{\rho p}. \label{vortex sheet dipole formula} 
  \end{equation}
  \end{enumerate}
\end{theorem}

Formulas \eqref{dipole formula} and \eqref{vortex sheet dipole formula} agree in the limit where $\omega$ and $\rho_+$ vanish. In this case they recover the formula
obtained by Wheeler in \cite{wheeler2016integral}; also see \cite{longuet1989capillary}. With nontrivial vorticity in the bulk, it becomes considerably more difficult to find clean expressions for $p$, as the definition of $V$ is not in any way related to the location of the interface $S$.  

The asymptotics in Theorem~\ref{decay theorem} and the identity in Theorem~\ref{integral identity theorem}(b) have the following corollary.

\begin{corollary}[Angular momentum] \label{angular momentum corollary} 
  Assume there exists a solitary wave with localized vorticity satisfying the assumptions of Theorem \ref{no excess mass theorem}.  If the total angular momentum 
  \[
    \left\{
      \begin{array}{ll} 
        \displaystyle  \int_\Omega x^\perp \cdot \rho u \, dx \quad & \text{if } n = 2, \\ \\
        \displaystyle \int_\Omega x \times \rho u \, dx  & \text{if } n = 3
      \end{array}
    \right.
  \]
  is finite then $p=0$. (Here $\rho \equiv 1$ in Case~\I.) Furthermore, for a vortex sheet every nontrivial wave has infinite angular momentum.
\end{corollary}

In Case~\I, owing to the more complicated formula \eqref{dipole formula} for $p$, we cannot conclude that waves with finite angular momentum are trivial. However, the fact that $p=0$ does imply that the leading-order terms for $u$ and $\eta$ in the asymptotics \eqref{eta n-d better decay} vanish.

While we do not pursue this here, all of the above results hold when the free surface $S$ is overhanging.  It is only necessary that $S$ is smooth, does not self-intersect, and is asymptotically flat in the sense that, outside of some bounded neighborhood of $0$, it is given as the graph of a function $\eta$ satisfying \eqref{eta decay}.

\section{Existence of waves with algebraic decay} \label{existence section}
In this section we prove Theorem~\ref{existence theorem} on the existence of steady water waves with localized vorticity and show that these waves exhibit the algebraic decay rates required by our asymptotic theory.   

\subsection{Splitting of \texorpdfstring{$u$}{u}} \label{splitting section}
When analyzing rotational water waves, it is often useful to decompose the velocity field into an irrotational part, given as the gradient of a potential $\varphi$, and a vortical part $V$:
\begin{equation}
  u = \nabla \varphi + V \qquad \textrm{in } \Omega,\label{varphi V splitting} 
\end{equation}
where $\varphi$ is harmonic and $V$ is divergence free.  How we choose to define $V$ will depend on the dimension.  Consider first the case $n =2$.  Then the fact that $u$ is divergence free implies that there exists a stream function $\psi_V$ with $V = \nabla^\perp \psi_V$.  In light of \eqref{varphi V splitting}, it must also be true that $\Delta \psi_V = \omega$, and hence $\psi_V$ can be expressed as the Newtonian potential of $\omega$ up to a harmonic function.   With that in mind, we define 
\begin{subequations} \label{n-d V def}
\begin{equation}
  V := \nabla^\perp \left[ \left(\frac{1}{\gamma_2} \log{|\cdot|}\right) * \omega - \frac{1}{\gamma_2} \varpi  \log{|\cdot -  \xi^* |} \right]  \qquad \textrm{if } n =2, \label{2-d V def} 
\end{equation} 
where 
\[ \varpi := \int_{\Omega} \omega_{\mathrm{ac}} \, dx + \sum_{i=1}^M \varpi^i \]
is the (total) vortex strength and $\xi^*$ is a fixed point in $\Omega^c$.  The second term in \eqref{2-d V def} can be thought of as a ``phantom vortex'' in the air region that counterbalances the total vorticity in the water.   By including it, we ensure that the splitting \eqref{varphi V splitting} decomposes $u$ as the sum of two $L^2$ functions on any closed subset of $\Omega \setminus \Xi$.  

In the three-dimensional case, the Green's function for the Laplacian enjoys better decay at infinity, and thus we have this desirable splitting property without  needing to introduce phantom vortices.  Indeed, we define $V$ directly using the Biot--Savart law,
\begin{equation}
  {V}(x) := \frac{1}{\gamma_3} \int_{\Omega} \omega(y) \times \frac{x-y}{|x-y|^3}  \, dy \qquad \textrm{if } n = 3.\label{3-d V def} 
\end{equation} 
Thanks to \eqref{3d vanishing assumption}, we show in Appendix~\ref{V appendix} that $\int_\Omega \omega\, dx = 0$; see \eqref{net vorticity zero}.
\end{subequations}

\subsection{Existence theory in Sobolev spaces}\label{sec sobolev}
Let us recall two results from \cite{shatah2013travelling} in some detail.  For each $k > 3/2$, there exists a one-parameter family of two-dimensional capillary-gravity solitary waves with a point vortex 
\[ \mathscr{C}_{\mathrm{loc}} = \left \{ (\eta(\varpi), u(\varpi), c(\varpi)) : |\varpi| < \varpi_0  \right\} \]
such that 
\[ \eta(\varpi) \in H_{\mathrm{e}}^k(\R),\qquad u(\varpi) - V(\varpi) \in H_{\mathrm{e}}^{k-1}(\Omega(\varpi)) \times H_{\mathrm{o}}^{k-1}(\Omega(\varpi)), \qquad \omega(\varpi) = \varpi \delta_{(0,-1)},\]
where $\Omega(\varpi)$ is the fluid domain corresponding to $\eta(\varpi)$, $\omega(\varpi)$ is the vorticity for the velocity field $u(\varpi)$, and $V(\varpi)$ is the function given by \eqref{n-d V def}.  The subscripts of `e' and `o' indicate evenness or oddness with respect to $x_1$, respectively.  This curve bifurcates from the absolutely trivial state of no motion:  $(\eta(0), u(0), c(0)) = (0,0,0)$,
and to leading order the solutions have the form 
\begin{equation} \label{point vortex leading order}  
  \begin{split} 
    &c(\varpi)  = \left(-\frac{\varpi}{2\gamma_2}  + o(\varpi^2) \right) e_1, \qquad \| u(\varpi) - V(\varpi) \|_{{H}^{k-1}(\Omega(\varpi))}  = O(|\varpi|^3), \\
    & \left\| \eta(\varpi) - \frac{\varpi^2}{4\pi^2} (g-\sigma \partial_{x_1}^2)^{-1} \left( \frac{x_1^2 - 1}{(1+x_1^2)^2} \right) \right\|_{H^k(\R)} = O(|\varpi|^3).   
  \end{split} 
\end{equation}

In the same paper, the authors construct a three-parameter family of traveling capillary-gravity waves with a vortex patch:
\[ \mathscr{S}_{\mathrm{loc}} = \big\{ (\eta,u,c)(\varpi,\rho,\tau) : |\varpi| < \varpi_0, ~ 0 < \rho <  \rho_0, ~ |\tau| < \tau_0 \big\}.\]
Here, $\varpi$ is the total vorticity, $\rho$ measures the approximate radius of $\supp{\omega(\varpi)}$, and $\tau$ arises due to a certain degeneracy in the linearized problem at the trivial state.   From now on we suppress the dependence on $\rho$ and $\tau$, which we hold fixed. These waves lie in $L^2$-based Sobolev spaces: the free surface profile  $\eta(\varpi) \in H^k_{\mathrm{e}}(\R)$, and the velocity field $u(\varpi) \in L^2_{\mathrm{e}}(\Omega(\varpi)) \times L^2_{\mathrm{o}}(\Omega(\varpi))$.  The vorticity $\omega(\varpi)$ has compact support $D(\varpi)$ that is a perturbation of a ball centered at $(0,-1)$:
 \[ \partial D(\varpi, \rho, \tau) = \left\{ \rho\left( \cos{\theta} + \tau \sin{(2\theta)}, \, \sin{\theta} - \tau \cos{(2\theta)} \right) + O(\rho^2(\rho+\varpi)) : \theta \in [0,2\pi) \right\}.\]
 Moreover, $u(\varpi) \in H^k(\Omega \setminus D)$ and $\omega(\varpi) \in H^1(D)$.  As in the point vortex case, this family bifurcates from the absolutely trivial state. One can show that these waves have the same leading order form as in \eqref{point vortex leading order}.

It is worth noting that, in fact, many families of the form $\mathscr{S}_{\mathrm{loc}}$ exist:  the argument in \cite{shatah2013travelling} allows one to select at the outset the value of the vorticity on each streamline in the patch $D$ from a fairly generic class of distributions.  For any such choice, there exists a corresponding $\mathscr{S}_{\mathrm{loc}}$; see \cite[Remark 2.2(b)]{shatah2013travelling}.  

The proofs in \cite{shatah2013travelling} work by applying implicit function-type arguments to nonlinear operators between Sobolev spaces. For point vorticies, the unknowns are $\eta,\varphi,c$ with $\varpi$ as a parameter, while for vortex patches there is an additional variable representing the shape of the patch, as well as two additional parameters $\rho,\tau$. In both cases the relevant linearized operators have an upper triangular structure which implies that $\eta,\varphi$ are (locally) uniquely determined by the remaining variables and parameters \cite[Lemma~4.1 and the proof of Theorem~2.1]{shatah2013travelling}.

While the above existence theory furnishes a leading-order description of these waves, it unfortunately does not say much about their pointwise behavior at infinity; knowing $\eta$ up to $O(\varpi^3)$ in $H^k(\R)$ is not sufficient to infer an explicit decay rate, much less the specific algebraic asymptotics claimed in \eqref{SWZ asymptotics}.

Our strategy in proving Theorem \ref{existence theorem} is to reconsider the local uniqueness of $\eta,\varphi$ in weighted H\"older spaces, again taking advantage of the upper triangular structure of the linearized equations. We treat $c$ and $\omega$ as given functions of the parameters, coming from the solutions in Sobolev spaces, and neglect the Euler equations \eqref{steady euler u} and advection condition \eqref{u vortex advection} inside the fluid. Since our weighted H\"older spaces are contained in the relevant Sobolev spaces, the local uniqueness (in both spaces) implies that the Sobolev and weighted H\"older solutions coincide.

\subsection{Change of variables} 
Fix any Sobolev regularity index  $k \geq 5$ and take either the family of waves with a point vortex $\mathscr{C}_{\mathrm{loc}}$ or the family of solitary waves with a vortex patch $\mathscr{S}_{\mathrm{loc}}$.  In the latter case, fix some (nonzero) values for $(\rho, \tau)$ and consider the corresponding curve lying in $\mathscr{S}_{\mathrm{loc}}$ parameterized by $\varpi$.  Note that by Morrey's inequality, $H^k(\R) \subset C^{3+\alpha}_\bdd(\R)$, for any $\alpha \in (0,1)$.

Let $\eta \in H_{\mathrm{e}}^k(\R)$ with $\Omega = \{ x \in \R^2 : x_2 < \eta(x_1) \}$ the corresponding fluid domain. We wish to study the vector fields $u :\Omega \to \R^2$ that can be written as
\begin{align}
  \label{u ansatz}
  u = \nabla^\perp \psi + V(\varpi)
\end{align}
for some harmonic function $\psi \in \dot{H}_{\mathrm{e}}^k(\Omega)$, where $V(\varpi)$ is the divergence free vector field defined according to \eqref{n-d V def} using the vorticity $\omega(\varpi)$ given by the family that we have selected and with $\xi^* = (1,0)$.   Note that the local uniqueness of solutions implies that $u$ is the velocity field for a solitary wave with wave velocity $c(\varpi)$ and vorticity $\omega(\varpi)$ if and only if $\eta = \eta(\varpi)$ and $\psi$ is a harmonic conjugate of $\varphi(\varpi)$.   Plugging the ansatz \eqref{u ansatz} for $u$ into the boundary conditions \eqref{u boundary cond} and using Bernoulli's law, we see that $\psi$ must satisfy the following elliptic system:
\begin{subequations} \label{existence: psi Eulerian equation}
  \begin{alignat}{2}
    \label{laplace}
    \Delta\psi &= 0 &\qquad& \textrm{in } \Omega, \\
    \label{dynamic}
    \frac 12 |\nabla^\perp\psi + \nabla^\perp \psi_V(\varpi) + c(\varpi)|^2 + g\eta -\sigma 
    \frac{\eta_{xx}}{(1+\eta_x^2)^{3/2}}
     &= 0 && \textrm{on } S,\\
    \label{kinematic}
    \psi + \psi_V(\varpi) - c_1(\varpi)\eta &= 0
    && \textrm{on } S,
  \end{alignat}
\end{subequations}
where $\psi_V(\varpi)$ is the stream function for $V(\varpi)$ in the sense that $\nabla^\perp \psi_V(\varpi) := V(\varpi)$. Note also that  \eqref{kinematic} results from taking the kinematic boundary condition in \eqref{u boundary cond}, re-expressing it in terms of $\psi$, and integrating once.  

Our first task is to make a change of coordinates to flatten the domain $\Omega$.  With that in mind, fix an even function $a \in C^\infty_\textup{c}(\R)$ supported in
$[-1,1]$ and with $a(0) = 1$.  We will work in the variables $(X,Y)$ defined
implicitly by
\begin{align*}
  x_1 = X,
  \qquad 
  x_2 = x_2(X,Y) = Y+\eta(X)a(Y).
\end{align*}
This is valid whenever $\n\eta_{L^\infty}
\n{a'}_{L^\infty} < 1$, and so it can be done for small-amplitude waves like those in $\mathscr{C}_{\mathrm{loc}}$ or $\mathscr{S}_{\mathrm{loc}}$. For finite-amplitude waves, we can simply extend the support of $a$, but this makes
the notation more cumbersome and so we will not pursue that here. It is easily seen that the mapping $(x_1, x_2) \mapsto (X,Y)$ sends the fluid domain $\Omega$ to the
half-space 
\[ R := \{(X,Y) \in \R^2 : Y < 0\}.\]
In particular, the image of the free surface $S$ is just $T := \{ (X,0) : X \in \R\}$.

It is tedious but elementary to show that \eqref{existence: psi Eulerian equation} is transformed to the following system:
\begin{subequations}\label{flattened}
  \begin{alignat}{2}
    \label{Laplace}
    \psi_{XX} + \psi_{YY} - f_1(\psi, \eta)&= 0  &\qquad& \textrm{in } R,\\
    \label{Dynamic}
    -c_1(\varpi)\psi_Y + g\eta - \sigma\eta_{XX} - f_2(\psi, \eta ,\varpi) &=   0&& \textrm{on } T,\\
    \label{Kinematic}
    \psi - c_1(\varpi)\eta + \psi_V(\eta,\varpi)&= 0 && \textrm{on } T,
  \end{alignat}
\end{subequations}
where the nonlinear terms in \eqref{flattened} are described by the mappings
\begin{align*}
  f_1(\psi,\eta) &:=
  -(a_Y \eta \psi_{YY} - a \eta_{XX} \psi_Y - a_{YY} \eta \psi_Y
  + 3 a_Y \eta \psi_{XX} - 2 a \eta_X \psi_{XY} )\\
  &\qquad- (a^2 \eta_X^2 \psi_{YY} - 2a a_Y \eta \eta_{XX} \psi_Y + 2 a a_Y
  \eta_X^2 \psi_Y + 3 a_Y^2 \eta^2 \psi_{XX} - 4 a a_Y \eta \eta_X
  \psi_{XY})\\
  &\qquad -\eta (a^2 a_Y \eta_X^2 \psi_{YY} -a a_Y^2 \eta \eta_{XX} \psi_Y -
  a^2 a_{YY} \eta_X^2 \psi_Y + 2a a_Y^2 \eta_X^2 \psi_Y \\
  & \qquad \qquad + a_Y^3 \eta^2
  \psi_{XX} - 2 a a_Y^2 \eta \eta_X \psi_{XY}),\\
  f_2(\psi,\eta,\varpi) &:= 
  -\tfrac 12 \left( V_1(\eta,\varpi)^2 - 2 \psi_Y V_1(\eta,\varpi) + V_2(\eta,\varpi)^2 + 2\psi_X V_2(\eta,\varpi)
  + \psi_Y^2 + \psi_X^2\right) 
  \\ & \qquad
  +\eta_X \psi_Y (V_2(\eta,\varpi) + \psi_X)
  - \tfrac 12 \eta_X^2 \psi_Y^2+ \sigma \left( (1+\eta_X^2)^{3/2} -1 \right) \eta_{XX}.
\end{align*}
We are abusing notation somewhat by writing 
\[ \psi_V(\eta,\varpi) := \psi_V(\varpi)(\cdot, \eta(\cdot)), \qquad V(\eta, \varpi) := V(\varpi)(\cdot, \eta(\cdot)).\]
It is also important to note that  $f_1(\psi, \eta)$ is supported in the slab $\{ -1 < Y < 0\}$ for any $(\psi, \eta)$.  In the next subsection, we will describe more precisely the domains and codomains of $f_1$ and $f_2$.

\subsection{Functional-analytic setting}

As mentioned above, we wish to reexamine the existence problem in weighted H\"older spaces introduced in Section \ref{notation section}.  Specifically, for the weight we will always take
\begin{align*}
  w = w(X,Y) := \frac{X^2+(1-Y)^2}{1-Y}.
\end{align*}
Note that $1/w$ is related to the Poisson kernel for a half-plane. It is easily confirmed that this weight is equivalent to $\jbracket x^2/\jbracket{x_2}$ in the original variables.

Observe that, since the vorticity has compact support, 
an argument much simpler than the proof of Lemma~\ref{V asymptotics lemma} shows that $\psi_V \in C_w^{\ell+\alpha}(T)$ for any $\ell \geq 0$. Another useful feature of these weighted spaces is that, if $f, g \in C^{\ell+\alpha}_w$, then $fg \in C^{\ell+\alpha}_{w^2}$.

Finally, it will be convenient in the coming analysis to have notation for slab subdomains of $R$.  For any $k > 0$, we denote by
\[ R_k := \left\{ (X,Y) \in \R^2 :  -k < Y < 0 \right\} \subset R\]
the strip of height $k$ having upper boundary $T$ and lower boundary $B_k := \{ Y =
-k\}$. 

With this notation in place, our objective is to find solutions $(\psi,\eta) \in C^{3+\alpha}_w(\overline R) \times C^{3+\alpha}_w(\R)$ of the flattened system \eqref{flattened} with $0 < |\varpi| \ll 1$.  
First, for each $\varpi \in \R$, $\eta \in C^{3+\alpha}_w(T)$, and $f
\in C^{1+\alpha}_w(\overline R)$ supported in $R_1$, we define
$\Psi = K(\eta,\varpi) f$ to be the unique bounded solution to the Dirichlet
problem
\begin{align*}
  \Delta\Psi = 0 \textup{~in~} R,
  \qquad 
  \psi = c_1(\varpi)\eta - \psi_V(\varpi) \textup{~on~} T.
\end{align*}
Then \eqref{flattened} can be written as the operator equation $\F(\psi,\eta,\varpi) = 0$ for $\mathcal{F} = (\mathcal{F}_1,\mathcal{F}_2)$ with
\begin{align}
  \label{F def}
  \begin{aligned}
    \F_1(\psi,\eta,\varpi) &:= 
    \psi - K(\eta,\varpi) f_1(\psi,\eta,\varpi),\\
    \F_2(\psi,\eta,\varpi) &:= 
    g\eta - \sigma \eta_{XX}  - c_1(\varpi)\psi_V(\eta,\varpi) - f_2(\psi,\eta,\varpi).
  \end{aligned}
\end{align}
For the time being, this must be understood in a purely formal sense. 

\subsection{Smoothness of \texorpdfstring{$\mathcal{F}$}{F}}

We begin by proving that $\F$ is a well-defined $C^1$ mapping from 
$C^{3+\alpha}_w(\overline R) \times C^{3+\alpha}_w(T) \times \R \to C^{3+\alpha}_w(\overline R)
\times C^{1+\alpha}_w(T)$. Inspecting the form of $f_1$, we see that the
following lemma is sufficient.
\begin{lemma}[Smoothness]\label{Ksmooth}
  For any function $b \in C^\infty(\R)$ with support in $[-1,0]$, the
  mapping
  \begin{align*}
    K_b \maps C^{1+\alpha}_w(\overline R) \times C^{3+\alpha}_w(T) \times \R 
    \to
    C^{3+\alpha}_w(\overline R),
    \qquad K_b(f,\eta,\varpi) 
    :=
    K(\eta,\varpi)(bf) 
  \end{align*}
  is of class $C^1$.  
\end{lemma}
The proof of Lemma~\ref{Ksmooth} relies on the following two lemmas on the decay properties of solutions to Dirichlet problems on the halfspace $R$.
\begin{lemma}\label{sunlemma}
  Let $h \in C^0_w(T)$. Then for any $k > 0$ the unique bounded
  solution $\Psi$ to 
  \begin{align*}
    \Delta \Psi = 0 \textup{~in~} R,
    \qquad 
    \Psi = h \textup{~on~} T
  \end{align*}
  obeys the estimate
  \begin{align*}
    \| \Psi \|_{C^4_w(R \setminus R_k)}
    \le C_k \| h \|_{C^0_w(T)}.
  \end{align*}
  \begin{proof}
    This follows from \cite[Lemma~3.1]{sun1997analytical}. Note that there is a minor misprint in the statement of that lemma:  $(1\pm \psi)^{\abs\beta-1}$ should be $\abs\psi^{\abs\beta-1}$. 
      \end{proof}
\end{lemma}

\begin{lemma}\label{simplegreens}
  Let $f \in C^{1+\alpha}_w(\overline R)$ with support contained in $R_1$. 
  Then for any $k > 1$ the unique bounded solution $\Psi$ to 
  \begin{align*}
    \Delta \Psi = f \textup{~in~} R,
    \qquad 
    \Psi = 0 \textup{~on~} T
  \end{align*}
  satisfies
  \begin{align*}
    \| \Psi \|_{C^{3+\alpha}_w(R_k)}
    \le C_k \| f \|_{C^{1+\alpha}_w(R)}.
  \end{align*}
\end{lemma}
\begin{proof}
  Defining $F \in C_w^{2+\alpha}(\overline{R})$ by
  \[ F(X,Y) := \int_{-1}^Y f(X,Z) \, dZ, \]
  we have that $F$ is supported in $R_1$ and satisfies $\partial_Y F = f$.  

  Now, we may express $\Psi$ in terms of $f$ via the Poisson kernel representation formula
  \begin{align*}
    \Psi(X,Y) &= \int_R f(W,Z) G(X,Y,W,Z) \, dW \, dZ,
  \end{align*}
  where 
  \[ G(X,Y,W,Z) :=  - \frac 1{4\pi} \left(\log{\left(|W-X|^2 + |Z-Y|^2\right)} - \log{\left(|W-X|^2+|Z+Y|^2\right)} \right).\]
  Integrating this by parts once in $Z$ yields 
  \[ \Psi(X,Y) = - \int_{R_1} F(W,Z) G_Z(X,Y,W,Z) \, dW \, dZ. \]
  Note that the integration domain is $R_1$ since this is where $F$ is supported.  We easily check that 
  \begin{align*}
    \partial_X^j G_Z,\ 
    [\partial_X^j G_Z]_\alpha
    \le \frac {C_k}{\abs{W-X}^2+1}
    \qquad 
    \text{for $(W,Z) \in R_1$, $Y=-k$, and $j \le 3$},
  \end{align*}
 which is enough to prove that 
  \begin{align*}
    \n\Psi_{C^{3+\alpha}_w(B_k)} \le C_k \n F_{C^0(R_1)}
    \le C_k \n f_{C^0(R_1)}.
  \end{align*}
  We conclude by applying \cite[Lemma~5.2]{craig1988symmetry}, which gives the following weighted estimate for the Poisson equation in a slab:
  \begin{align}
    \label{eqn:dirstrip}
    \n \Psi_{C^{3+\alpha}_w(R_k)}  &\le 
    C_k\left( \n\Psi_{C^{3+\alpha}_w(T)} 
    + \n\Psi_{C^{3+\alpha}_w(B_k)} 
    + \n{\Delta \Psi}_{C^{1+\alpha}_w(R_k)}
    \right)\\
    \notag
    &\le C_k \n f_{C^{1+\alpha}_w(R_1)}.
    \qedhere
  \end{align}
\end{proof}

\begin{proof}[Proof of Lemma~\ref{Ksmooth}.]
  For $f \in C^{1+\alpha}_w(\overline R)$ and $\eta \in C^{3+\alpha}_w(T)$, we know that $K_b(f,\eta,\varpi)$ exists as an element of $C^{3+\alpha}(\overline R)$. It remains to show that it is also in $C^{3+\alpha}_w(\overline R)$ and that it depends smoothly on its arguments. 
  
  Towards that end, we decompose 
  \begin{align*}
    K_b(f,\eta,\varpi) = K_{b,1} f  + K_{b,2}(\eta,\varpi),
  \end{align*}
  where $K_{b,1} f := \Psi_1$ is the unique bounded $C^{3+\alpha}(\overline R)$ solution to 
  \begin{align*}
    \Delta \Psi_1 = bf \textup{~in~} R,
    \qquad \Psi_1 = 0 \textup{~on~} T,
  \end{align*}
  and $K_{b,2}(\eta,\varpi) := \Psi_2$  is the unique
  bounded $C^{3+\alpha}(\overline R)$ solution to
  \begin{align*}
    \Delta \Psi_2  = 0 \textup{~in~} R,
    \qquad \Psi_2 = c_1(\varpi)\eta - \psi_V(\eta,\varpi) \textup{~on~} T.
  \end{align*}

  Applying Lemma~\ref{simplegreens} to $\Psi_1$ yields the estimate
  \begin{align*}
    \| \Psi_1 \|_{C^{3+\alpha}_w(R_4)} 
    \le C \n{bf}_{C^{1+\alpha}_w(R)}
    \le C \n{f}_{C^{1+\alpha}_w(R)}.
  \end{align*}
  In particular, this gives us control of $\Psi_1$ restricted to the line $B_2 \subset R_4$, and hence Lemma~\ref{sunlemma} can be used to bound $\Psi_1$ on the half-space $R\setminus R_3$:
  \begin{align*}
    \n{\Psi_1}_{C^4_w(R \setminus R_3)}
    \le C\n{\Psi_1}_{C^0_w(B_2)}
    \le C\n{\Psi_1}_{C^{3+\alpha}_w(R_4)}
    \le C\n f_{C^{1+\alpha}_w(R)}.
  \end{align*}
  Combining these two estimates yields
  \begin{align*}
    \n{\Psi_1}_{C^{3+\alpha}_w(R)} \le C \n f_{C^{1+\alpha}_w(R)}. 
  \end{align*}

  On the other hand, we obtain from Lemma~\ref{sunlemma} the following bound for $\Psi_2$: 
  \begin{align*} 
    \n{\Psi_2}_{C^4_w(R \setminus R_1)} \le C \left(\n{\eta}_{C^0_w(T)} + \n{\psi_V(\eta,\varpi)}_{C^0_w(T)}\right).
  \end{align*}
  Thus $\Psi_2$ is controlled in a half-space positively separated from $T$. To control it on the remainder of $R$, we again use the estimate \eqref{eqn:dirstrip} from \cite[Lemma~5.2]{craig1988symmetry} to obtain
  \begin{align*}
    \n {\Psi_2}_{C^{3+\alpha}_w(R_2)}  &\le 
    C\left( \n{\Psi_2}_{C^{3+\alpha}_w(T)} + \n{\Psi_2}_{C^{3+\alpha}_w(B_2)} \right)\\  
    &\le
    C\left( \n{\eta}_{C^{3+\alpha}_w(T)} + \n{\psi_V(\eta,\varpi)}_{C^{3+\alpha}_w(T)} + \n{\Psi_2}_{C^{3+\alpha}_w(B_2)} \right).  
  \end{align*}
  As $B_2 \subset R\setminus R_1$, the preceding two bounds can be combined to show
  \begin{align*}
    \n{\Psi_2}_{C^{3+\alpha}_w(R)} \le C\left( \n{\eta}_{C^{3+\alpha}_w(T)} + \n{\psi_V(\eta,\varpi)}_{C^{3+\alpha}_w(T)} \right).
  \end{align*}
  Thus $K_b$ is indeed well-defined as a
  mapping $C^{1+\alpha}_w(\overline R) \times C^{3+\alpha}_w(T) \times \R \to
  C^{3+\alpha}_w(\overline R)$. Since $K_{b,1}$ is linear, the above
  estimates show that it is also smooth. Similarly $K_{b,2}$ is the composition
  of a linear map and the smooth map $C_w^{3+\alpha}(T) \times \R \to C^{3+\alpha}_w(T)$ given by $(\eta,\varpi) \mapsto c_1(\varpi)\eta - \psi_V(\eta,\varpi)$.
\end{proof}

\subsection{Improved decay for \texorpdfstring{$\eta_X$}{eta\_X}}

So far we have been working with $\eta \in C^{3+\alpha}_w(T)$, which only implies $\eta,\eta_X = O(1/X^2)$. In this subsection we show that $\eta_X = O(1/\abs X^{2+\varepsilon})$ so that \eqref{eta decay} holds. The proof will rely on the fact that $c = O(\varpi)$ implies that $\sigma > |c|^2/4g$ for $0 < \abs\varpi \ll 1$. 

\begin{lemma} \label{improved derivative decay lemma} 
Suppose that there exists a two-dimensional capillary-gravity solitary wave with localized vorticity and $\sigma > \abs c^2/4g$ such that the velocity field has the decomposition
 \[ u = \nabla^\perp \psi + V, \qquad \psi \in C_w^{3+\alpha}(\Omega),\]
 and the free surface $\eta \in C_w^{3+\alpha}(\R)$. Then the derivatives of the free surface profile $\eta$ enjoy the improved decay 
 \begin{equation}
   \partial_{x^\prime} \eta = O\left( \frac{1}{|x^\prime|^{2+\varepsilon} }\right), \qquad \textrm{as } |x^\prime| \to \infty.
   \label{higher derivatives better decay eta} 
 \end{equation}
\end{lemma}
\begin{proof}
Inverting equation \eqref{Dynamic} we know that 
\begin{equation*}
\eta = G \ast f_2,
\end{equation*}
where $G = G(X)$ has the Fourier transform
\begin{equation*}
\widehat{G}(k) = \frac{1}{\sigma k^2 - |c||k| + g}.
\end{equation*}
We subtract off from $G$ the Green's kernel for $g - \sigma \partial^2_X$, and write 
\begin{equation*}
\eta = G_1 \ast f_2 + (g - \sigma \partial^2_X)^{-1} f_2,
\end{equation*}
where
\begin{align*}
(g - \sigma \partial^2_X)^{-1} f_2 & = {1\over 2\sqrt{\sigma g}} e^{-\sqrt{g/\sigma} |\cdot|} \ast f_2 =: G_2 * f_2, \\
\widehat{G}_1(k) & = {|c||k| \over (\sigma k^2 - |c||k| + g) (\sigma k^2 + g)}.
\end{align*}

Since convolution with $G_2$ and $G_2'$ preserves the algebraic decay of any function it acts on, we know that 
\begin{equation}\label{decay 1}
   G_2 * f_2 ,
  G_2' * f_2 
  \in
  C_{w^2}^{0} (\mathbb R).
\end{equation}
Examining the form of $G_1$, we find
\begin{equation*}
\left\{\begin{array}{llll}
G_1(X) = O(1), \ \ & G'_1(X) = O(1)   \quad &  \text{for $X$ small}, \\
G_1(X) = O(1/|X|^2), \ & G'_1(X) = O(1/|X|^3) \quad &  \text{for $X$ large}.
\end{array}\right.
\end{equation*}
In particular,
\begin{equation*}
\langle X \rangle^{2+\varepsilon} \partial^j_X G_1 = O(1/|X|^{1-\varepsilon}) \quad \text{for } X \text{ large.} 
\end{equation*}
Denote $r(X) := \langle X \rangle^{2+\varepsilon}$. Then we see that
\begin{align*}
\left| r G_1' \ast f_2 \right| & = \left| \int_{\mathbb R} r(X) G_1'(W) f_2(X-W) \, dW \right| \\
& \le \int_{\mathbb R} \left| G_1'(W) {r(X) \over r(X-W)} \right| \left| f_2(X-W) r(X-W) \right| \, dW\\
  & \le \sup_{X,W} \frac{r(X)}{r(X-W) r(W)}  \| r G_1' \|_{L^p} \|r f_2\|_{L^q} \\
  &\lesssim \| r G_1' \|_{L^p} \|f_2^{q-2} r^{q}\|^{1/q}_{L^\infty} \|f_2\|_{L^2}^{2/q}.
\end{align*}
Note that now 
\[
r G_1' \in L^p(\mathbb R) \quad \text{for any } p > {1\over 1 - \varepsilon}.
\]
On the other hand, $f_2 \in L^2$ and 
\[
\|r^q f^{q-2}_2\|^{1/q}_{L^\infty} = \left\| r^{q/(q-2)} f_2 \right\|^{(q-2)/q}_{L^\infty} \le \|f_2\|^{(q-2)/q}_{C^0_{w^2}}, \quad \text{ for } \quad  {q \over q - 2} \le {4 \over 3 + \varepsilon}.
\]

Therefore we have
\begin{equation}\label{range q}
\varepsilon < {1\over q} \le {1- \varepsilon \over 8},
\end{equation}
which implies that $\varepsilon < {1\over 9}$. This way choosing $p, q$ that satisfy \eqref{range q} we have
\begin{equation}
  \label{this one}
  r G_1' = \langle X \rangle^{2+\varepsilon} G_1' \in L^\infty.
\end{equation}

Writing
\begin{align*}
  \eta &= G_1 * f_2 + G_2 * f_2, 
  \qquad 
  \eta_X = G_1' * f_2 + G_2' * f_2.
\end{align*}
and using \eqref{decay 1} and \eqref{this one} we conclude that
\begin{equation*}
\langle X \rangle^{2+\varepsilon} \eta_X \in C_{\bdd}(\R).\qedhere
\end{equation*}
\end{proof}

\subsection{Proof of existence}

Having shown that $\F$ is smooth and well-defined, we can at last complete the argument leading to Theorem \ref{existence theorem}.
\begin{proof}[Proof of Theorem \ref{existence theorem}]
Following our discussion above, we will in fact prove statements (a) and (b) simultaneously.  From Lemma~\ref{Ksmooth}, we have seen that the mapping $\mathcal{F}: C^{3+\alpha}_w(\overline R) \times C^{3+\alpha}_w(T) \times \R \to C^{3+\alpha}_w(\overline R) \times C^{1+\alpha}_w(T)$ defined in \eqref{F def} is $C^1$.   

It is also evident that the Fr\'echet derivative $D_{(\psi,\eta)}\F(0,0,0)$ is a linear isomorphism from $C^{3+\alpha}_w(\overline R) \times C^{3+\alpha}_w(T)$ to $C^{3+\alpha}_w(\overline R) \times C^{1+\alpha}_w(T)$. Indeed, the nonlinear terms in $f_1$ and $f_2$ are quadratic or higher, and so 
    \begin{align*}
      D_{(\psi,\eta)}\F(0,0,0)
      &= 
      \begin{pmatrix}
        \mathrm{id}  & 0 \\
        0 & g-\sigma \partial_X^2
      \end{pmatrix}.
    \end{align*}
    Because our weight $w$ is algebraic, $g-\sigma \partial_X^2$ is invertible as a mapping
    $C^{3+\alpha}_w (T) \to C^{1+\alpha}_w(T)$, with its inverse given
    as convolution with the exponentially decaying kernel $G_2$ from Lemma \ref{improved derivative decay lemma}.
    Thus the full operator matrix $D_{(\psi,\eta)}\F(0,0,0)$ is 
    invertible.
    
    We can now apply the implicit function theorem to $\mathcal{F}$ at the trivial solution $(0,0,0)$ to deduce the existence of a local curve of solutions in $C^{3+\alpha}_w(\overline R) \times C^{3+\alpha}_w(T) \times \R$.  These weighted H\"older spaces lie inside the Sobolev spaces from Section~\ref{sec sobolev} with $k=3$, and so the local uniqueness for $\eta,\psi$ in Sobolev spaces implies that they agree with those constructed in \cite{shatah2013travelling}. In particular, the corresponding $(\eta,u,c)$ solve the full problem \eqref{steady euler plus bc} and, in the point vortex case, the advection condition \eqref{u vortex advection}.

   Next, Lemma~\ref{improved derivative decay lemma} ensures that, for each $(\eta,u,c)$ in this family, $\eta_X$ decays faster as stated in \eqref{eta decay}.  Since $\eta$ is small in $C^{3+\alpha}_w(\R)$, Remark~\ref{decay of phi remark} implies that $\varphi = o(1)$ and hence by Theorem \ref{decay theorem} that $\eta$ has the improved decay \eqref{higher derivatives better decay eta}. Lastly, the leading-order expressions for $\eta$ and $u$ at infinity given in  \eqref{SWZ asymptotics} are obtained by using \eqref{dipole formula} to compute the dipole moment $p$, and inserting that into \eqref{n-d better decay}.  This last straightforward calculation is left to Section~\ref{example section}.  
\end{proof}

\section{Asymptotic properties}

The purpose of this section is to show that the weaker localization assumption \eqref{decay assumptions} implies the improved the decay and asymptotic form claimed in Theorem~\ref{decay theorem}. Our strategy is to split the velocity field as in Section \ref{splitting section} and then separately investigate $\varphi$ and $V$.

\subsection{Asymptotic form of \texorpdfstring{$V$}{V}}
We begin with an elementary lemma that establishes the asymptotics for $V$ under the vorticity localization assumptions in \eqref{localized vorticity assumptions}.  

\begin{lemma}[Asymptotics for $V$ with localized vorticity] \label{V asymptotics lemma} 
Suppose that $\omega$ satisfies \eqref{localized vorticity assumptions}, and let $V$ be defined as in \eqref{n-d V def}. Then 
\begin{equation}
  \label{V asymptotics}
  V = -{1\over \gamma_n}\nabla \left( \frac{m \cdot x}{|x|^n} \right) + O\left( {1\over |x|^{n+\varepsilon}} \right)
  \qquad 
  \text{as $\abs x \to \infty$},
\end{equation}
where $\gamma_n$ is the surface area of the unit sphere in $\R^n$ and $m$ is the vortex impulse \eqref{vortex impulse}.
\end{lemma}

Statements like \eqref{V asymptotics} are well known for compactly supported vorticity, see for instance the formal proof in 
\cite[Chapter 3.2]{saffman1992book}.  Similar computations are commonplace in electrostatics where the term dipole originates.  The main point of Lemma \ref{V asymptotics lemma}, therefore, is that the asymptotics are the same under the weaker localization assumptions in \eqref{localized vorticity assumptions}.  Because it is somewhat technical, the proof is deferred to Appendix \ref{V appendix}. 

 As a simple corollary, we detail the decay rates of two quantities that are important for the analysis in the next section; cf.\ \eqref{kinematic condition} and \eqref{bernoulli condition}.

 \begin{corollary} \label{boundary term asymptotics corollary} 
 Under the hypotheses of Lemma \ref{V asymptotics lemma}, 
 if $\eta$ has the decay \eqref{eta decay}, it follows that 
\begin{equation}
  c \cdot V|_S = -\frac{1}{\gamma_n} c^\prime \cdot \nabla \left( \frac{m^\prime \cdot x^\prime}{|x^\prime|^n} \right) + O \left(\frac{1}{|x^\prime|^{n+\varepsilon}} \right) \label{cV asymptotics} 
\end{equation}
and 
\begin{equation}
  \label{NV asymptotics}
  |x|^n N \cdot V|_S = -{m\cdot e_n\over \gamma_n } + O\left( {1\over |x'|^{\varepsilon}} \right),  
\end{equation}
as $|x^\prime| \to \infty$.  
\end{corollary}
\begin{proof}  
The first equation \eqref{cV asymptotics} follows immediately from Lemma \ref{V asymptotics lemma} and the decay assumptions on $\eta$ in \eqref{eta decay}.  For the second, we compute that
\begin{align*}
  |x|^n N \cdot V|_S 
  &= 
  -\frac 1{\gamma_n} 
  \left( \abs{x'}^n + O\!\left( \frac 1{\abs{x'}^{n-1+\varepsilon}} \right) \right)
  \left( e_n + O\!\left(\frac 1{\abs{x'}^{n+\varepsilon}} \right) \right)\\
  &\qquad\qquad\qquad\cdot
  \left( \frac m{\abs{x'}^n} - n \frac{(m'\cdot x')x'}{\abs{x'}^{n+2}}
  + O\!\left( \frac 1{\abs{x'}^{n(n+1+\varepsilon)}} \right)\right)\\
  &= - \frac{m \cdot e_n}{\gamma_n} + O\!\left( \frac 1{\abs{x'}^{n+\varepsilon}} \right),
\end{align*}
where we have again used the decay assumptions \eqref{eta decay}. 
\end{proof}

\subsection{Asymptotic form of \texorpdfstring{$\varphi$}{phi}}

Now we turn to the potential for the irrotational part of the flow. Consider first localized vorticity of Case~\I. Recall that we split the velocity $u = \nabla \varphi + V$ where the vortical part $V$ is defined by \eqref{n-d V def}. The velocity potential $\varphi$ can be recovered from $\eta$ and $V$ as the unique solution to the elliptic problem
\begin{subequations} \label{varphi elliptic problem}
\begin{alignat}{2}
 \Delta \varphi &= 0 && \qquad \textrm{in } \Omega  \label{varphi harmonic} \\
 N \cdot \left( \nabla \varphi +  V \right) & = c \cdot N && \qquad \textrm{on } S, \label{kinematic condition} 
 \end{alignat}
 which vanishes as $|x| \to \infty$.  From the Euler equations \eqref{steady euler plus bc}, we see that $\varphi$ must in addition satisfy the nonlinear boundary condition
  \label{governing equations} 
\begin{alignat}{2}
 \frac{1}{2} | \nabla \varphi + V|^2 - c \cdot \left( \nabla \varphi + V \right) + g \eta & = -\sigma \nabla \cdot N && \qquad \textrm{on } S. \label{bernoulli condition}
 \end{alignat}\end{subequations}
Here we have applied Bernoulli's law on $S$ and evaluated the pressure using the dynamic boundary condition from \eqref{u boundary cond}.  This is justified by our assumption in \eqref{3d vanishing assumption} that $\omega$ is tangential to the free surface.  

As it is harmonic, we clearly have $\varphi \in C^\infty(\Omega)$.  To confirm the regularity of $\varphi$ up to the boundary, we first notice from \eqref{vorticity assumptions} and classical potential theory that $V \in C^{\alpha}(\overline{\Omega})$ for any $\alpha \in (0,1)$; see, for example, \cite[Chapter 4]{gilbarg2001elliptic}.
Now, under the assumption that $\eta \in C_\bdd^2(\R^{n-1})$ and decays according to \eqref{eta decay}, the Neumann data for $\varphi$ in \eqref{kinematic condition} is of class $C^\alpha(S)$, for all $\alpha \in (0,1)$.  Applying elliptic regularity theory, these deductions at last lead us to the conclusion that 
\begin{equation}
  \varphi \in C^{1+\alpha}(\overline{\Omega}), \qquad \textrm{for all } \alpha \in (0,1). \label{phi C^1} 
\end{equation}  

Finally, we observe that the decay rate assumed for $u$ in \eqref{u decay} together with the asymptotics \eqref{V asymptotics} for $V$ imply that the potential satisfies 
\begin{equation}
  \varphi = o\left( \frac{1}{|x|^{n-2}}\right), \qquad \textrm{as } |x| \to \infty.\label{varphi decay} 
\end{equation}
Indeed, in two dimensions, this asserts only that $\varphi$ vanishes at infinity, which follows from the discussion above.  For the three-dimensional case, \eqref{varphi decay} would be a direct consequence of the fundamental theorem of calculus if $\eta$ were identically zero; the modification of this argument to allow for nontrivial $\eta$ only requires that $\eta$ and $\nabla\eta$ be uniformly bounded.

\begin{lemma} \label{dipole lemma} 
Suppose that there exists a solitary wave with Case~\I\ vorticity
and the decay \eqref{decay assumptions}.
Then there exists a dipole moment $q = (q', m_n/\gamma_n)$ such that 
\begin{equation}
  \varphi(x) = \frac{q \cdot x}{|x|^n} + O\left(\frac{1}{|x|^{n-1+\varepsilon}} \right), \qquad \nabla \varphi(x) = \nabla \left( \frac{q\cdot x}{|x|^n} \right) + O\left(\frac{1}{|x|^{n+\varepsilon}} \right), \label{phi asymptotic} 
\end{equation}
as $|x| \to \infty$.
\end{lemma}
\begin{proof}
The main idea behind this argument is that the asymptotic properties \eqref{phi asymptotic} can be determined via elliptic theory from the kinematic boundary condition, the equation satisfied by $\varphi$, and the decay assumption \eqref{decay assumptions}.  In comparison to the irrotational case considered in \cite{wheeler2016integral}, there are additional terms coming from the vortical contribution $V$ that must be understood using the asymptotic information contained in Lemma \ref{V asymptotics lemma} and Corollary \ref{boundary term asymptotics corollary}.  

  With that in mind, recall that the Kelvin transform $\tilde\varphi = \tilde\varphi(\tilde x)$ of $\varphi = \varphi(x)$ is defined by 
\[ \tilde{x} := \frac{x}{|x|^2}, \qquad \tilde\varphi (\tilde x) := \frac{1}{|\tilde x|^{n-2}} \varphi \left( \frac{\tilde x}{|\tilde x|^2} \right).  \]
Denote by $\Omega^{\sim}$ the image of $\Omega \setminus B_1(0)$ under the map $x \mapsto \tilde x$: 
\[ \Omega^\sim := \left\{ \tilde{x} \in \R^n : \frac{\tilde{x}}{|\tilde x|^2} \in \Omega \setminus B_1(0) \right\}.\]
The asymptotic behavior of $\varphi$ as $|x| \to \infty$ is determined by the behavior of $\tilde \varphi$ in a neighborhood of the origin in the $\tilde x$-variables.  Notice that by \eqref{eta decay}, we have $0 \in \partial \Omega^\sim$, and \eqref{varphi decay} ensures that $\tilde\varphi (0) = 0$ and $\tilde \varphi$ can be extended to the boundary as a $C^0(\overline{\Omega^\sim})$ class function.  

Similarly as in \cite[Appendix A]{wheeler2016integral}, we find that the decay assumptions on $\eta$ we made in \eqref{eta decay} guarantee that $\partial \Omega^\sim$ is $C^2$ in a neighborhood of $0$; let $S^\sim \subset \partial \Omega^\sim$ be a small portion of the boundary containing $\tilde x =0$.  

As the Kelvin transform of a harmonic function is harmonic, we know that $\Delta \tilde \varphi = 0$ in $\Omega^\sim$.  We now show that the kinematic equation \eqref{kinematic condition} leads to an oblique boundary condition for $\tilde \varphi$.  A simple calculation shows that $\tilde N$, the normal vector to $\Omega^\sim$ at $\tilde x \in S^\sim$, is related to $N$ by the formula 
\[ \tilde N (\tilde x) = N(x) - 2 \left(\frac{N(x) \cdot x}{|x|^2} \right) x.\]
We then find that
\[
N \cdot \left(c - V \right) = N \cdot \nabla \varphi = -(n-2) \frac{x \cdot N}{|x|^n} \tilde \varphi + \frac{1}{|x|^n} \tilde{N} \cdot \nabla \tilde \varphi,
\]
or, equivalently,
\begin{equation}
  \tilde N \cdot \nabla \tilde\varphi + \tilde a \tilde\varphi = \tilde b \qquad \textrm{on } S^\sim \setminus \{0\},
  \label{transform boundary condition}
\end{equation}
where $\tilde a = \tilde a(\tilde x)$ and $\tilde b = b(\tilde x)$ are given by
\begin{equation}
  \tilde a(\tilde x) := -(n-2)\left( N \cdot x \right), \qquad \tilde b(\tilde x) := {|x|^n} N \cdot (c - V). 
  \label{def alpha beta} 
\end{equation}
Our objective is to use elliptic theory to infer that $\tilde \varphi$ has the desired H\"older regularity in a neighborhood of the origin in the Kelvin transform variables.  Naturally, this requires us to establish the $C^{\varepsilon}(S^\sim)$ H\"older continuity of the coefficients $\tilde a$ and $\tilde b$ above. This can be achieved from the (stronger) H\"older regularity and the decay properties of $\tilde a,\tilde b$ as functions of the untransformed variable $x$.  In particular, we apply Lemma \ref{lem inversion} with $\alpha = \varepsilon,\ \beta = 2\varepsilon$, and $k = 4\varepsilon$. Moreover from the decay of $\eta$ in \eqref{eta decay} and asymptotics of $V \cdot N|_S$ obtained in  \eqref{NV asymptotics} it follows that 
\begin{equation*}
\tilde a(0) = 0, \quad \tilde b(0) = {m_n \over \gamma_n}.
\end{equation*}

The proof of the lemma is now essentially complete.  Using the regularity of $\varphi$ in \eqref{phi C^1}, we can argue as in \cite[Appendix A]{wheeler2016integral} to show that $\tilde \varphi$ is an $H^1(\Omega^\sim)$ weak solution of the Laplace equation with oblique boundary condition \eqref{transform boundary condition}.  Elliptic regularity theory then implies that $\tilde{\varphi} \in C^{1+\varepsilon}(\Omega^\sim \cup S^\sim)$ (see, for example, \cite[Theorem 5.51]{lieberman2013oblique}), and therefore that it admits the expansion 
\[ 
  \tilde \varphi(\tilde x) = q \cdot \tilde x + O(|\tilde x|^{1+\varepsilon}),
  \qquad 
  \nabla \tilde \varphi(\tilde x) = q + O(|\tilde x|^\varepsilon),
  \]
where $q := \nabla \tilde \varphi(0)$.  Finally, evaluating the transformed boundary condition \eqref{transform boundary condition} at $\tilde x = 0$, we infer that 
  \begin{align*}
    \frac{m_n}{\gamma_n}
    =
    \tilde b
    =
    \tilde N \cdot \grad \tilde\varphi  + \tilde a \tilde \varphi
    = e_n \cdot q
  \end{align*}
  and hence $q = (q^\prime, m_n/\gamma_n)$.  Returning to the original variables, this confirms that $\varphi$ has the asymptotic form \eqref{phi asymptotic}.
\end{proof}

\begin{remark} \label{decay of phi remark} Examining the above argument, it is clear that the localization assumption on $u$ in \eqref{decay assumptions} is needed only insofar as it implies that $\varphi$ decays according to \eqref{varphi decay}.  In two dimensions, this can be removed altogether when $\eta$ has sufficiently small Lipschitz constant.  This is a consequence of the fact that, in this setting, the unique solution of the elliptic system \eqref{varphi elliptic problem} is given by the single-layer potential for the Neumann data $(c-V) \cdot N$.
But an easily calculation shows that $\int_S (c-V) \cdot N\, dS = 0$, which implies that $\varphi=o(1)$. 
\end{remark}

Next, consider the situation for the vortex sheet.  As the flow is irrotational in the air and water, we have $u_\pm = \nabla \varphi_\pm$, where $\varphi_\pm$  will then satisfy
  \begin{subequations} \label{vortex sheet governing equations} 
\begin{alignat}{2}
 \Delta \varphi_\pm &= 0 && \qquad \textrm{in } \Omega_\pm  \label{vortex sheet varphi harmonic} \\
 N_\pm \cdot \nabla \varphi_\pm & = c \cdot N_\pm && \qquad \textrm{on } S \label{vortex sheet kinematic condition} \\
 \frac{1}{2} \jump{\rho | \nabla \varphi|^2} - \jump{\rho c \cdot \nabla \varphi} + g \jump{\rho} \eta & = \pm \sigma \nabla \cdot N_\pm && \qquad \textrm{on } S. \label{vortex sheet bernoulli condition} \end{alignat}
\end{subequations}  
Again, the assumptions on $u$ and $\eta$ in \eqref{decay assumptions} ensure that $\varphi_\pm$ have the decay rate \eqref{varphi decay}.

\begin{corollary} \label{vortex sheet dipole}
Suppose that there exists a traveling wave solution with Case~\VS~vorticity and the decay \eqref{decay assumptions}.  Then there exist dipole moments $p_\pm = (p_\pm^\prime, 0)$ such that 
\begin{equation}
  \varphi_\pm(x) = \frac{p_\pm \cdot x}{|x|^n} + O\left(\frac{1}{|x|^{n-1+\varepsilon}} \right), \qquad \nabla \varphi_\pm(x) = \nabla \left( \frac{p_\pm \cdot x}{|x|^n} \right) + O\left(\frac{1}{|x|^{n+\varepsilon}} \right), \label{phi asymptotic sheet} 
\end{equation}
as $|x| \to \infty$.  
\end{corollary}
\begin{proof}  This proof works almost identically to that of Lemma \ref{dipole lemma} but without the difficulties related to $V$.  Let $\Omega_\pm^\sim$ denote the image under $x \mapsto \tilde x$
  of $\Omega_\pm \setminus B_1(0)$.  Likewise, let $\tilde \varphi_\pm$ be the Kelvin transform of $\varphi_\pm$.  
 Once again, the vanishing of $\eta$ at infinity \eqref{eta decay} implies that $0 \in \partial \Omega_\pm^\sim$, and the decay of $\varphi_\pm$ gives $\tilde\varphi_\pm (0) = 0$.  It follows that each $\tilde \varphi_\pm$ can be extended to the boundary as a $C^0(\overline{\Omega_\pm^\sim})$ class function.

Furthermore, we can easily confirm that $\partial \Omega_\pm^\sim$ is $C^2$ in a neighborhood of $0$.  Let $S^\sim \subset \partial \Omega_\pm^\sim$ be a portion of the mutual boundary containing $\tilde x =0$.  The transformed potentials $\tilde\varphi_\pm$ are harmonic in $\Omega_\pm^\sim$ and satisfy the oblique boundary conditions 
\[ \tilde N_\pm \cdot \nabla \tilde \varphi_\pm + \tilde a_\pm \tilde \varphi_\pm = \tilde b_\pm \qquad \textrm{on } S^\sim \setminus \{0\},\]
where $\tilde N_\pm$ are the images of $N_\pm$ under the Kelvin transform, and 
\[ \tilde a_\pm(\tilde x) := -(n-2) \left(\frac{N_\pm \cdot \tilde x}{|\tilde x|^2} \right), \qquad \tilde b_\pm(\tilde x) := \frac{1}{|\tilde x|^n} N_\pm \cdot c_\pm.\]
Observe that the elliptic problems for $\tilde \varphi_\pm$ are essentially decoupled.  Arguing exactly as in the irrotational case, we can show that $\tilde a_\pm $ and $\tilde b_\pm$ are uniformly $C^\varepsilon$ in a neighborhood of the origin.  Elliptic regularity then implies the existence of the dipole moments $p_\pm =: \tilde \varphi_\pm(0)$, and the fact that $p_{n\pm} = 0$ follows once more from evaluating the boundary condition at $\tilde x =0$ but noting that $\tilde b(0) = 0$ in this case.  
\end{proof}

\subsection{Asymptotic forms of \texorpdfstring{$\eta$}{eta} and \texorpdfstring{$u$}{u}}

Having determined that $\varphi$ and $V$ are dipoles at infinity, we are now prepared to prove our theorem characterizing the asymptotic forms of $\eta$ and $u$.  

\begin{proof}[Proof of Theorem \ref{decay theorem}]  First consider the statement in part (a).  Solving for $\eta$ in the Bernoulli condition \eqref{bernoulli condition}, we find that 
\begin{align*}
\eta(x^\prime) & = \frac{1}{g} \left( c\cdot \nabla \varphi + \frac{1}{2} |\nabla \varphi|^2 - \sigma \nabla \cdot N + c \cdot V + \nabla \varphi \cdot V + \frac{1}{2}|V|^2 \right) \\
&= \frac{1}{g|x^\prime|^n} \left( c \cdot q - n \frac{(c^\prime \cdot x^\prime)( q^\prime \cdot x^\prime)}{|x^\prime|^2} \right) + \frac{1}{g} \left( c \cdot V + \nabla \varphi \cdot V + \frac{1}{2}|V|^2 \right) + O\left( \frac{1}{|x^\prime|^{n+\varepsilon}} \right),
\end{align*}
where the second line follows from the decay assumed on $\eta$ in \eqref{eta decay} and our estimate of $\varphi$ in \eqref{phi asymptotic}.  Now, from \eqref{V asymptotics} we know that the $\nabla \varphi \cdot V$ and $|V|^2$ are $O(1/|x^\prime|^{n+\varepsilon})$. 
  Inserting the leading-order formula for $c \cdot V$ derived in \eqref{cV asymptotics} then yields
\[ \eta = \frac{1}{g|x^\prime|^n} \left( c^\prime \cdot \left( q^\prime - {m^\prime\over\gamma_n} \right) -n \frac{(c^\prime \cdot x^\prime)\left( (q^\prime - {m^\prime \over \gamma_n}) \cdot x^\prime \right)}{|x^\prime|^2 }  \right) + O\left( \frac{1}{|x^\prime|^{n+\varepsilon}} \right). \]
Defining $p := q - m/\gamma_n$, this is exactly the claimed asymptotic expression for $\eta$ in \eqref{eta n-d better decay}.  Note that because $q_n = m_n/\gamma_n$, it is indeed true that $p_n = 0$.  Likewise, the asymptotic form of $u$ stated in \eqref{u n-d better decay} simply follows from writing $u = \nabla \varphi + V$ and using the dipole formula for $V$ in \eqref{V asymptotics} and for $\nabla \varphi$ in \eqref{phi asymptotic}.  

The argument for (b) is similar.  From \eqref{vortex sheet bernoulli condition}, we find that 
\begin{align*}
\eta(x^\prime) & = \frac{1}{g\jump{\rho}} \left( \jump{\rho c\cdot \nabla \varphi + \frac{1}{2} \rho |\nabla \varphi|^2} - \sigma \nabla \cdot N  \right) \\
&= \frac{1}{g\jump{\rho} |x^\prime|^n} \left( c \cdot \jump{\rho p} - n  \frac{(c^\prime \cdot x^\prime)( \jump{\rho p^\prime} \cdot x^\prime)}{|x^\prime|^2} \right) + O\left( \frac{1}{|x^\prime|^{n+\varepsilon}} \right),
\end{align*}
which implies \eqref{vortex sheet eta better decay}.  The asymptotics for $u_\pm$ asserted in \eqref{vortex sheet u better decay} were already proved in Corollary \ref{vortex sheet dipole}, since $u_\pm = \nabla \varphi_\pm$.  
\end{proof}

\section{Nonexistence and the dipole moment formula}

\subsection{Nonexistence}
First, we establish that there exist no waves with localized vorticity in Case~\I\ or Case~\VS\ having a single-signed free surface profile.  In fact, we prove the stronger statement that all such waves must have no excess mass in the sense that $\int \eta\,dx =0$.   

In the two-dimensional setting, we will need the following result on the configuration of the streamlines in a neighborhood of a point vortex.  
 
 \begin{lemma}[Streamlines] \label{streamlines lemma} 
   Suppose that $n=2$ and that there exists a solitary wave with Case~\PV~vorticity and the decay \eqref{decay assumptions}. Fix $\alpha \in (0,1)$.  For each vortex center $\xi^i \in \Xi$, and $\delta > 0$ sufficiently small, there exists an open connected set $\tilde{B}_\delta^i \subset \Omega$ with $\xi^i \in \tilde{B}_\delta^i$, and $\partial \tilde{B}_\delta^i$ is a closed integral curve of $u-c$ that admits the global parameterization
 \begin{equation}
   \label{parameterization tilde B} \partial \tilde{B}_\delta^i = \{ (\tilde{r}_\delta^i(\theta) \cos(\theta), \, \tilde r_\delta^i(\theta) \sin(\theta)) : \theta \in [0,2\pi) \},
 \end{equation}
 where $\tilde r_\delta^i \in C^{1+\alpha}$ is a $2\pi$-periodic function with $\tilde r_\delta^i(0) = \delta$, $\partial_\theta \tilde r_\delta^i = O(\delta^3)$.
  \end{lemma}
 \begin{proof}
As we are only concerned with local properties of the flow, we may without loss of generality suppose that $\xi^i = 0$.  
In light of \eqref{u vortex advection} and \eqref{2-d V def}, we know that
\[ c = \nabla \varphi(0) + V_{\mathrm{ac}}(0) + \sum_{\substack{j = 1 \\ j \neq i}}^M V^j(0) + V_{\mathrm{p}}(0),\]
where $V_{\mathrm{ac}}$, $V^j$, and $V_{\mathrm{p}}$ are the contributions of the absolutely continuous part of the vorticity, the $j$-th point vortex, and the phantom vortex, respectively.  Note that $\varphi$, $V^j$, and $V_{\mathrm{p}}$ are each harmonic near the origin, whereas $V_{\mathrm{ac}} \in C^{\alpha}$, since $\omega_{\mathrm{ac}} \in L^\infty(\Omega)$. Then we may write
\begin{align}\label{velocity near vortex} u(x) - c &=:  \frac{1}{2\pi} \varpi^i \nabla^\perp \log{|x|} + G(x), \end{align}
where $G(0) = 0$, and in a neighborhood of the origin, $G$ belongs to $C^\alpha$ and is divergence free in the distributional sense.       
We can therefore introduce a function $\Psi$ of class $C^{1+\alpha}$ near the origin such that $\nabla^\perp \Psi = G$ and $\Psi(0) = 0$.  It follows that the level sets of the function
\begin{equation}\label{Hamiltonian} H(x) := \frac{1}{2\pi} \varpi^i \log{|x|} + \Psi(x)\end{equation}
coincide locally with the integral curves of $u - c$.  Let  $\delta_0 > 0$ be sufficiently small so that 
\[ \frac{x}{|x|} \cdot \nabla H(x) = \frac{\varpi^i}{2\pi |x|} + \frac{x^\perp}{|x|} \cdot G(x) > 0  \qquad \textrm{in } B_{\delta_0}(0) \setminus \{0\}.\]
Then $H$ is strictly increasing in the radial direction on this punctured ball.  We define the neighborhoods $\tilde B_\delta^i$ to be the super level sets of $H$.  Writing \eqref{Hamiltonian} in polar coordinates and applying the implicit function theorem then yields the parameterization function $\tilde r_\delta^i$.  \end{proof}
 
 In the next lemma, we establish a key integral identity that is a consequence of Bernoulli's theorem and the localization of the vorticity \eqref{localized vorticity assumptions}.

 \begin{lemma} \label{int (c-u)omega lemma} 
   Suppose that there exists a solitary wave with Case~\NS~vorticity and the decay \eqref{decay assumptions}. Then, in the the two-dimensional case,
 \begin{equation}
   \int_{\Omega} (u -c) \omega \, dx  = 0, \label{2-d (c-u)omega identity} 
 \end{equation} 
  and in the three-dimensional setting
  \begin{equation}
    \int_{\Omega}  \left( u - c \right) \times \omega   \, dx = 0. \label{3-d (c-u)omega identity}
  \end{equation}
 For localized vorticity in Case~\PV, \eqref{2-d (c-u)omega identity} holds with $\omega_{\mathrm{ac}}$ in place of $\omega$. 

 \end{lemma}
 \begin{proof}
 First observe that, working in three dimensions, the Euler equations lead to
  \begin{align}
    \label{Bernoulli}
    (u-c) \times \omega
    = \nabla\left( \frac 12 |u-c|^2 + P + gx_3 - \frac 12 |c|^2 \right),
  \end{align} 
  which holds in the sense of distributions on $\Omega$.   As the left-hand side above is in $L^\infty(\Omega)$, we have that  
 \[ B(x) := \frac{1}{2} |u-c|^2 + P + g x_n - \frac{1}{2} |c|^2 \in W^{1,\infty}(\Omega).\]
 The identity \eqref{Bernoulli} ensures that the weak tangential derivative of $B$ vanishes on any smooth integral curve of $u-c$.  In particular, this applies to the free surface, and from \eqref{u boundary cond} and \eqref{eta decay} we infer that $B$ vanishes identically on $S$.  Now, taking $R > 0$ large and integrating over $\Omega \cap B_R(0)$ using \eqref{Bernoulli}, we find that 
\begin{equation}
 \label{3-d int (u-c)omega} 
 \begin{split}
   \int_{\Omega \cap B_R(0)}  (u-c) \times \omega \, dx & = \int_{\partial B_R(0) \cap \Omega} B N \, dS.
 \end{split}
\end{equation}

An analogous identity can be derived in two dimensions.   Suppose that there are point vortices in the flow, as this can be easily adapted to the case of non-singular localized vorticity.  The Euler equations once again imply that
\begin{equation}
  \label{2-d Bernoulli}
  (c - u)^\perp \omega = \nabla \left( \frac{1}{2} |u-c|^2 + P + g x_2 - \frac{1}{2} |c|^2 \right),
\end{equation}
in the distributional sense on $\Omega \setminus \Xi$.  Note that we are free to replace $\omega$ by $\omega_{\mathrm{ac}}$ above, as they agree away from $\Xi$.   For $R > 0$ and $\delta > 0$ sufficiently small, define the domain 
\begin{equation}\label{domain}  \Omega_{R, \delta} := \left( \Omega \cap B_R(0) \right) \setminus \bigcup_{i=1}^M \tilde B_\delta^i,\end{equation}
where the sets $\tilde B_\delta^i$ are those described in Lemma \ref{streamlines lemma}.  From \eqref{2-d Bernoulli}, we have that
\begin{equation}
 \label{2-d int (u-c)omega} 
 \begin{split}
   \int_{ \Omega_{R,\delta}} (u - c)^\perp \omega_{\mathrm{ac}} \, dx & = \int_{\Omega \cap \partial B_R(0)} B N \, dS.
 \end{split}
\end{equation}
Observe that there are no boundary integral terms over the sets $\partial \tilde B_\delta^i$ because they are smooth closed integral curves of $u-c$, and hence $B$ is constant along them according to the above discussion. Likewise, $B$ vanishes on the free surface and thus there is no integral over $S$ occurring in \eqref{2-d int (u-c)omega}.  

To finish the argument, we will show that $B \in L^1(\Omega)$, which guarantees that there exists a sequence of radii $\{R_j\}$ with $R_j \to \infty$ and such that the integrals on the right-hand sides of \eqref{3-d int (u-c)omega} and \eqref{2-d int (u-c)omega} vanish as $j \to \infty$.    Observe that the identities \eqref{Bernoulli} and \eqref{2-d Bernoulli}, together with the localization assumption \eqref{localized vorticity assumptions}, imply that 
\[ \nabla B \in L^1(\Omega)\cap L^\infty(\Omega), \quad |x|^k \nabla B \in L^1(\Omega), \]
where recall that $k > n^2$.  Our main tool for translating estimates in weighted Sobolev spaces to $L^1$ is the Caffarelli--Kohn--Nirenberg inequality \cite{caffarelli1984,catrina2001}, which states that
\begin{equation}\label{CKN}
\left( \int_{\Omega} |x|^{-bs}|B|^s\ dx \right)^{2/s} \lesssim \int_{\Omega} |x|^{-2a} |\nabla B|^2 \ dx
\end{equation}
for any $a$, $b$, and $s$ satisfying the relations
\begin{align*}
 a \in (-\infty, 0),~ b \in (a, a+1],~ s = {2\over b-a} &  \qquad \text{if } n = 2, \\
a \in (-\infty, 1/2), ~ b \in [a, a+1], ~ s = {6 + 2(b-a)} & \qquad \text{if } n = 3.
\end{align*}

In three dimensions, this gives
\begin{equation*}
\left( \int_{\Omega} |x|^{3k} |B|^6 \ dx \right)^{1/3} \lesssim \int_{\Omega} |x|^k |\nabla B|^2\, dx \le \|\nabla B\|_{L^\infty}\int_{\Omega} |x|^k |\nabla B|\, dx.
\end{equation*}
Thus, from H\"older's inequality and the above estimate we see that
\begin{align*}
\| B \|_{L^1(\Omega)} & \lesssim \left( \int_{\Omega} (1+|x|^{3k}) |B|^6 \, dx \right)^{1/6} < \infty,
\end{align*}
where we have applied the Gagliardo--Nirenberg--Sobolev inequality to control $B$ in $L^6(\Omega)$.

The argument in two dimensions is similar.  Taking $a = -k/2$, $b = -(k-1)/2$, and $s = 4$, we infer from \eqref{CKN} that
\begin{equation*}
\left( \int_{\Omega} |x|^{2(k-1)} |B|^4 \, dx \right)^{1/2} \lesssim  \|\nabla B\|_{L^\infty}\int_{\Omega} |x|^k |\nabla B|\, dx < \infty.
\end{equation*}
On the other hand, choosing $a = -1/2$, $b = 0$, and $s = 4$ yields
\begin{equation*}
\left( \int_{\Omega} |B|^4 \, dx \right)^{1/2} \lesssim \int_{\Omega} |x| |\nabla B|^2 \,  dx \le  \left\| |\nabla B|^{2 - {1\over k}} \right\|_{L^{{k\over k-1}}}\int_{\Omega} |x|^k |\nabla B|\, dx,
\end{equation*}
and so together these estimates furnish the bound
\begin{align*}
\| B \|_{L^1(\Omega)} & \lesssim \left( \int_{\Omega} \left(1+|x|^{2(k-1)}\right) |B|^4 \, dx \right)^{1/4} < \infty.
\end{align*}
Here we have used H\"older's inequality and relied on the fact that $k > 4$. 
 \end{proof}
 
 With the additional understanding of the streamlines near a point vortex given by Lemma \ref{streamlines lemma}, and the identities \eqref{2-d (c-u)omega identity}--\eqref{3-d (c-u)omega identity}, we can now prove the nonexistence theorem.  
 
 \begin{proof}[Proof of Theorem \ref{no excess mass theorem}]
   We begin with Case~\I\ vorticity.  Consider the vector field  $A$ defined by 
  \begin{align*}
    A := u_n(c-u) + \left( \frac{1}{2} |u|^2  - c \cdot u \right) e_n,
  \end{align*} 
  with domain $\Omega$ for Case~\NS\ and $\Omega \setminus \Xi$ for Case~\PV.
  It is easy to compute that 
  \[ \nabla \cdot A = \left\{ \begin{array}{ll} (c_1 - u_1) \omega & \textrm{if } n = 2 \\
  e_n \cdot \left(\omega \times (c-u) \right) & \textrm{if } n =3, \end{array} \right.\]
in the distributional sense, while the Bernoulli condition and \eqref{u boundary cond} together imply that
  \begin{align*}
    N \cdot  A 
    = \left( \frac{1}{2} |u|^2 - c \cdot u \right) N \cdot e_n 
    = \frac{1}{\jbracket{\nabla \eta}} (-g \eta - \sigma \nabla \cdot N) \qquad \textrm{on } S.
  \end{align*}

Our plan will be to apply the divergence theorem to $A$.  The most sensitive argument is needed for Case~\PV, so we treat that scenario first.  Let $\Omega_{R, \delta}$ be the domain defined in \eqref{domain}.  Integrating $\nabla \cdot A$ over $\Omega_{R,\delta}$ furnishes the identity
 \begin{equation}
   \begin{split}
     \int_{\Omega_{R,\delta}} \left( c_1 - u_1 \right) \omega \, dx & = g\int_{B_R(0) \cap S} \eta \, dx' + \sigma \int_{\partial B_R(0) \cap S} \nu \cdot N \, ds \\  & \qquad + \sum_i \int_{\partial \tilde B^i_\delta}  A \cdot N \, dS + \int_{\partial B_R(0) \cap \Omega}  A \cdot N \, dS,
   \end{split} \label{2-d no mass identity} 
\end{equation}
where in the second integral on the right-hand side we have used \eqref{u boundary cond} and then integrated by parts; $\nu$ and $ds$ refer to the normal vector and arc-length element with respect to the projection of $\partial B_R(0) \cap S$ onto $\mathbb R^{n-1} \times \{0\}$, respectively.  We know from Lemma \ref{int (c-u)omega lemma} that as $R \to \infty$ and $\delta \to 0$, the left-hand side will vanish.  On the other hand, from Lemma \ref{V asymptotics lemma} and Lemma \ref{dipole lemma}, we see that 
  \[ A = O(|u|) = O\left(|\nabla \varphi| + |V|\right) =  O\left(\frac{1}{|x|^n}\right), \qquad \textrm{as } |x| \to \infty,\]
  and hence that the  integral over $\partial B_R(0) \cap \Omega$ on the right-hand side of \eqref{2-d no mass identity} vanishes in the limit $R \to \infty$. Similarly, our assumptions on the decay of $\eta$ in \eqref{eta decay} guarantee that the integral over $\partial B_R(0) \cap S$ vanishes as $R \to \infty$. 

  Consider now the third term on the right-hand side in \eqref{2-d no mass identity}.  Without loss of generality, let us take $\Xi = \{ 0 \}$.  By construction, $u-c$ is tangent to $\partial \tilde B_\delta^i$, and hence 
\begin{align*}
\int_{\partial \tilde B^i_\delta} A \cdot N \, dS & =  \int_{\partial \tilde B^i_\delta}  \left[ {1\over 2}|u|^2 - c\cdot u \right] N_n \, dS = {1\over 2} \int_{\partial \tilde B^i_\delta} \left( |\nabla^\perp H|^2 - |c|^2 \right) N_n \, dS \\
& = {1\over 2} \int_{\partial \tilde B^i_\delta} \left| {\varpi^i \over 2\pi} {x^\perp\over |x|^2} + G \right|^2 N_n \, dS \\
& = {1\over 2} \int_{\partial \tilde B^i_\delta} \left[ \left( {\varpi^i \over 2\pi} \right)^2 {1\over |x|^2} + {\varpi^i \over \pi} {x^\perp \cdot G \over |x|^2} +|G|^2 \right] N_n \, dS.
\end{align*}
Recall that the function $G\in C^{\alpha}(\tilde B^i_\delta)$, for any $\alpha \in (0,1)$, and satisfies $G(0) = 0$,  so in particular $G(x) = O(|x|^{\alpha})$. The explicit parameterization \eqref{parameterization tilde B} shows that $\diam{\tilde{B}_\delta^i} = O(\delta)$, and therefore 
\[
\int_{\partial \tilde B^i_\delta} \left( {\varpi^i \over \pi} {x^\perp \cdot G \over |x|^2} +|G|^2 \right) N_n \, dS \longrightarrow 0 \quad \text{as } \ \delta \to 0.
\]
From Lemma \ref{streamlines lemma}, we know that $|x| = \tilde r_\delta^i = \delta + O(\delta^3)$ on $\partial \tilde B_\delta^i$, and so a simple calculation reveals that 
\begin{align*}
\int_{\partial \tilde B^i_\delta} {N_n \over |x|^2} \, dS = \int_{\partial \tilde B^i_\delta} N_n\left( {1 \over \delta^2} + O(1) \right) \, dS = O(\delta).
\end{align*}
Putting together all of the above deductions, we conclude that
\[
  \int_{\partial \tilde B^i_\delta}  A \cdot N \, dS \longrightarrow 0 \quad \text{as } \delta \to 0.
\]
 
  Finally, returning to \eqref{2-d no mass identity} and taking $\delta \to 0$ and $R \to \infty$, we find      
  \[
  \lim_{R\to \infty} \int_{B_R(0) \cap S} \eta\, dx' 
  = \int_{\mathbb R} \eta \, dx' 
  = 0,
  \]
  which completes the argument for the two-dimensional setting.      

  Next, consider non-singular localized vorticity in $\mathbb R^3$.  Applying the divergence theorem to $A$ on $\Omega \cap B_R(0)$, we obtain
  \begin{equation}
    \begin{split}
      \label{eqn:returntothis}
      \int_{B_R(0) \cap \Omega}
      e_n \cdot (\omega \times (c-u))\, dx
      & =
      g\int_{B_R(0) \cap S} \eta\, dx'
      + \sigma \int_{\partial B_R \cap S} N \cdot \nu\, ds \\
      & \qquad     + \int_{\partial B_R(0) \cap \Omega}  A \cdot N\, dS.
    \end{split} 
\end{equation}
  In view of Lemma \ref{int (c-u)omega lemma}, this implies that 
  \begin{align*}
    \label{eqn:returnedtothis}
    g\int_{B_R \cap S} \eta\, dx'
    + \sigma \int_{\partial B_R(0) \cap S} N \cdot \nu\, ds
    + \int_{\partial B_R(0) \cap \Omega}  A \cdot N\, dS
    \longrightarrow 0 \text{~as~} R \to \infty.
  \end{align*}
  Thanks again to \eqref{u n-d better decay}, we have $A = O(|u|) = O(1/|x|^n)$, and so the remaining integral over $\partial B_R \cap \Omega$ also vanishes as $R \to \infty$, leaving us with $\int \eta \, dx^\prime = 0$.
  
  Lastly, the argument for the case of a vortex sheet is a simpler version of that given above.  The vector field $A$ is divergence free (in the classical sense) in both the air and water regions, and its normal trace is continuous over $S$.  An application of the divergence theorem as above yields $\int \eta \, dx^\prime = 0$.  
  \end{proof}

\subsection{Dipole moment formula}
The objective of this section is to derive the formula \eqref{dipole formula} relating the dipole moment to the vortex impulse.   Following Wheeler \cite{wheeler2016integral}, our strategy is based on identifying a vector field whose divergence gives this energy-like quantity and which decays at infinity in such a way that we can recover $p$ upon integrating by parts.  The presence of vorticity significantly complicates this task.

\begin{proof}[Proof of Theorem \ref{integral identity theorem}] 
First consider part (a).  Let $A$ be the vector field 
\begin{equation}
 \label{AV def} 
 \begin{split} 
   A &:= \left( \varphi - {m \cdot x\over \gamma_n |x|^n} \right) (u-c) + (c \cdot x) u,
 \end{split} 
\end{equation}
with domain $\Omega$ for the non-singular localized vorticity case, or $\Omega \setminus \Xi$ if there are point vortices. It is easy to compute that
\begin{equation}
  \label{first dipole: div A_V} \begin{split}
    \nabla \cdot A & = u \cdot \nabla \varphi + c \cdot V - {u - c \over \gamma_n} \cdot \nabla\left( {m \cdot x \over |x|^n} \right) \\
    & = |u|^2 - (u-c) \cdot \left( V +  {1 \over \gamma_n } \nabla\left( {m \cdot x \over |x|^n} \right) \right).  
  \end{split} 
\end{equation}

Fix $\delta > 0$ and $R > 0$ and let $\Omega_{R, \delta}$ be given as in \eqref{domain}.  Applying the divergence theorem to $A$ on $\Omega_{R,\delta}$ leads to the identity
\begin{align*}
\int_{\Omega_{R,\delta}} \left( u \cdot \nabla \varphi + c \cdot V \right) \, dx & = \int_{B_R(0) \cap S} A \cdot N \, dS +  \sum_i \int_{\partial \tilde B^i_\delta} A \cdot N \, dS \\
& \qquad + \int_{\partial B_R(0) \cap \Omega} A \cdot N \, dS =: \mathbf{I} + \mathbf{II} + \mathbf{III}.
\end{align*}

Because $u \cdot N = c \cdot N$ on $S$, we have that 
\begin{equation}
  \label{first dipole: I vanishes} 
  A \cdot N= (c \cdot N)(c \cdot x) \qquad \textrm{on } S. 
\end{equation} 
Therefore, 
\begin{align*} \mathbf{I} & = -\int_{B_R(0) \cap S} \frac{c^\prime \cdot \nabla \eta}{\jbracket{\nabla \eta}} (c \cdot x) \, dS  = \int_{B_R(0) \cap S} \left( \frac{|c|^2}{\jbracket{\nabla \eta}} \eta - \nabla_S \cdot \left( \eta (c \cdot x) c\right) \right) \, dS,
\end{align*} 
where $\nabla_S \cdot$ denotes the surface divergence on $S$.  The second term in the integrand is a total derivatives of quantities vanishing at infinity, and thus
\[ \mathbf{I} \to |c|^2 \int_{\R^{n-1}} \eta(x^\prime) \, dx^\prime \qquad \textrm{as } R \to \infty.\]
We already established in Theorem \ref{no excess mass theorem} that waves of this type have no excess mass, and hence $\mathbf{I} \to 0$ as $R \to \infty$.

Next, consider $\mathbf{II}$.  For Case~\NS\ in either the two- or three-dimensional setting, the asymptotic information contained in Lemma \ref{V asymptotics lemma} in particular guarantees that $V \in L^2(\Omega)$, and hence $\mathbf{II}$ will simply vanish in the limit as $\delta \to 0$.  Likewise, in Case~\PV, the same will be true for  $V_{\mathrm{ac}}$.  To understand the contributions of the point vortices, note that we have from \eqref{velocity near vortex}, and \eqref{Hamiltonian} that $u - c = \nabla^\perp H$ around each point vortex, and thus $(u - c) \cdot N |_{\partial \tilde B^i_\delta} = T \cdot \nabla H = 0$, which leads to 
\begin{align*}
\mathbf{II} & = \sum_i \int_{\partial \tilde B^i_\delta} (c \cdot x) (c  \cdot N ) \, dS 
 = O(\delta).
 \end{align*}
  
Finally, to compute $\mathbf{III}$, we note that from the asymptotic formulas \eqref{V asymptotics} and \eqref{phi asymptotic},
\[
V = O(1/R^n), \quad \varphi = O(1/R^{n-1}), \quad \nabla \varphi = O(1/R^n), \quad \text{ on }\ \partial B_R(0) \cap \Omega.
\]
Thus
\begin{align*}
\mathbf{III} & = -\int_{\partial B_R(0) \cap \Omega} \left[ \left( \varphi - {m \cdot x\over \gamma_n |x|^n}  \right) (c \cdot N) - (c \cdot x) N \cdot ( \nabla \varphi + V)  \right] \, dS  + O\left(\frac{1}{R^n} \right).
\end{align*} 
Hence, as $R \to \infty$, $\mathbf{III}$ approaches the constant value 
\begin{align*} 
\mathbf{III} & \to \int_{\partial B_R(0) \cap \{ x_n < 0 \}} \left( -\left(\frac{q \cdot x}{|x|^n} - {m \cdot x\over \gamma_n |x|^n} \right) \frac{c \cdot x}{|x|} + (c \cdot x) \frac{x}{|x|} \cdot \nabla \left( \frac{q \cdot x}{|x|^n} - \frac{m \cdot x}{\gamma_n |x|^n}   \right) \right) \, dS    \\ 
& =   \int_{\partial B_R(0) \cap \{ x_n < 0 \}} \left( -n  \frac{(c \cdot x) (q \cdot x) }{|x|^{n+1}} + n  \frac{(c \cdot x)(m \cdot x)}{\gamma_n |x|^{n+1}}  \right)  \, dS  \\
& = -n \int_{\partial B_1(0) \cap \{x_n < 0\}} (c \cdot x) \left( q - \frac{m}{\gamma_n}  \right) \cdot x \, dS =   -\frac{\gamma_n}{2}  c \cdot \left( q - \frac{m}{\gamma_n }  \right)  = -\frac{\gamma_n}{2}  c \cdot p.
\end{align*}  

The argument for (b) is a slight variation of that given above.  Let us redefine $A$ to be the vector field 
\[ A :=   \rho \varphi (\nabla \varphi -c) + \rho (c \cdot x) \nabla \varphi,\]
which is smooth in $\Omega = \Omega_+ \cup \Omega_-$.  Applying the divergence theorem to $A$ on the set $B_R(0) \setminus S$ then gives
\begin{equation}
  \begin{split} \int_{B_R(0) \setminus S} \rho |\nabla \varphi|^2 \, dx &= \int_{B_R(0) \cap S} \left( A_{+} \cdot N_+ + A_{-} \cdot N_- \right) \, dS + \int_{\partial B_R(0) \setminus S} A \cdot N \, dS \\ 
    & =: \mathbf{I} + \mathbf{II}. \end{split} \label{vs div identity} 
\end{equation}
To evaluate $\mathbf{I}$, we use the kinematic boundary condition \eqref{vortex sheet kinematic condition} to infer that 
\[ A_{\pm} \cdot N_\pm = \rho_\pm (\nabla \varphi_\pm -c) \cdot N_\pm + \rho_\pm (c\cdot x) \nabla\varphi_\pm \cdot N_\pm = \rho_\pm (c \cdot x) c \cdot N_\pm \qquad \textrm{on } S,\]
hence 
\begin{align*}
\mathbf{I} &= -\jump{\rho} \int_{S \cap B_R(0)} (c\cdot x) c \cdot N_- \, dS \\
& = -\jump{\rho} \int_{S\cap B_R(0)} \left( \frac{|c|^2}{\jbracket{\nabla \eta}} \eta - \nabla_S \cdot \left( \eta (c \cdot x) c \right) \right) \, dS,
\end{align*}
which vanishes in the limit as $R \to \infty$ in view of Theorem \ref{no excess mass theorem}.  On the other hand, 
\begin{align*}
\mathbf{II} &= \int_{\partial B_R(0) \cap \Omega_+} A \cdot N \, dS + \int_{\partial B_R(0) \cap \Omega_-} A \cdot N \, dS \to \frac{\gamma_n}{2} \rho_+  p_+ \cdot c - \frac{\gamma_n}{2} \rho_-  p_- \cdot c,
\end{align*}
as $R \to \infty$.  Combining this with \eqref{vs div identity} gives the vortex sheet dipole formula in \eqref{vortex sheet dipole formula}, completing the proof.
\end{proof}

With Theorems~\ref{decay theorem} and \ref{integral identity theorem} in hand, we can now prove Corollary~\ref{angular momentum corollary} on the angular momentum.
\begin{proof}[Proof of Corollary \ref{angular momentum corollary}]
  Since the differences between the $n=2$ and $n=3$ are merely notational, we only give the three-dimensional argument.
From the asymptotic expansion of $u$ in \eqref{u n-d better decay}, we 
conclude that as $R \to \infty$, 
\begin{align*}
  \int_{\partial B_R(0)  \cap \Omega} 
  x \times u\, dS
  = 
  \int_{\partial B_R(0) \cap \{x_n < 0\}}
  x \times \nabla \left(\frac {p\cdot x}{|x|^n}\right)\, dS + O\left( \frac{1}{R^{1+\varepsilon}}\right).
\end{align*}
It follows that
\begin{align*} 
  \int_{\partial B_R(0) \cap \Omega} x \times u \, dS &\to -p \times \int_{\partial B_1(0) \cap \{x_n < 0\}} x \, dS = \frac{\pi^{\frac{n-1}{2}}}{\Gamma(\frac{n+1}{2})} p \times e_n, 
\end{align*}
as $R \to \infty$. Clearly, if the right-hand side above does not vanish, then the integral 
\mbox{$\int_\Omega x \times u\, dx$}
describing the total angular momentum will be divergent.  It follows that $p^\prime = 0$ is a necessary condition for the angular momentum to be finite.  On the other hand, we already know from Lemma \ref{dipole lemma} that $p_n = 0$, and hence $p$ must vanish identically.  The proof for Case~\I\ vorticity in two dimensions is identical, and hence omitted.  

For Case~\VS\ in either the two- or three-dimensional settings, the argument is the same.  Note that if $\jump{\rho p} = 0$, however, then the dipole moment formula \eqref{vortex sheet dipole formula} implies that $u_\pm \equiv 0$, meaning that the wave must be trivial.
\end{proof}

\subsection{Dipole moment for \texorpdfstring{$\mathscr{C}_{\mathrm{loc}}$}{C\_loc} and \texorpdfstring{$\mathscr{S}_{\mathrm{loc}}$}{S\_loc}} 
\label{example section} Finally, as an example application of Theorem \ref{integral identity theorem}, in this subsection we determine the dipole moment $p$ to leading order for the families $\mathscr{C}_{\mathrm{loc}}$ and $\mathscr{S}_{\mathrm{loc}}$ constructed in Theorem \ref{existence theorem}.

First, observe that given the leading-order forms of $\eta(\varpi)$ and $c(\varpi)$ detailed in \eqref{point vortex leading order} for $\mathscr{C}_{\mathrm{loc}}$ and $\mathscr{S}_{\mathrm{loc}}$, it suffices to compute all integrals on the lower half-plane $\{x_2 < 0 \}$.  In fact, it is enough to simply consider $\mathscr{C}_{\mathrm{loc}}$, as the vortex patches in $\mathscr{S}_{\mathrm{loc}}$ limit to the point vortices as the radius of the patch $\rho$ is sent to $0$.    Looking at the left-hand side of \eqref{dipole formula}, we anticipate that the highest-order term is 
\begin{align*}  
  \mathbf{I} &:= \int_{\{x_2 < 0\}} \left( c \cdot V + \frac{1}{\gamma_2} c \cdot \nabla \left( \frac{m \cdot x}{|x|^2} \right) \right) \, dx \\
  &= -\frac{c_1 \varpi}{\gamma_2} \int_{\{x_2 < 0\}} \partial_{x_2} \left(  \log{|x+e_2|} -  \log{|x-e_2|} \right) \, dx 
  + \frac{c_1 \varpi}{\gamma_2} \int_{\{x_2 < 0 \}} \partial_{x_1} \left( - \frac{2 x_1}{|x|^2} \right) \, dx, 
\end{align*}
where we are taking $\xi = (0,-1)$, $\xi^* = (0,1)$, so that $m = -2 e_1$.  After an elementary argument, we find that $\mathbf{I} = - c_1 \varpi$.

Now, using \eqref{point vortex leading order}, we know that 
\[ \| V \|_{L^2(\Omega)} = O(\varpi), \qquad \| \nabla \varphi \|_{L^2} = O(\varpi^3), \qquad c = \left(-\frac{\varpi}{2 \gamma_2} + o(\varpi^2) \right) e_1,\]
and hence \eqref{dipole formula} becomes 
\begin{align*}
-\frac{\gamma_2}{2}  c \cdot p & = \int_{\Omega} \left[ |u|^2 - u \cdot \left( V + \frac{1}{\gamma_2} \nabla \left( \frac{m \cdot x}{|x|^2} \right) \right)  \right] \, dx + \mathbf{I} \\
& = \int_{\Omega} \left[ \nabla \varphi \cdot V + |\nabla \varphi|^2 -  \frac{1}{\gamma_2}  \nabla \varphi \cdot \nabla \left( \frac{m \cdot x}{|x|^2} \right)   \right] \, dx + \mathbf{I} \\
& = \mathbf{I} + O(\varpi^4) = \frac{\varpi^2}{2 \gamma_2} + O(\varpi^4).
\end{align*}
Thus, for $|\varpi| \ll 1$, we find that 
\[ p = \frac{2\varpi}{\gamma_2} + O(\varpi^2).\]
Taking this value for $p$ in \eqref{n-d better decay}, we arrive at the asymptotic expressions in  \eqref{SWZ asymptotics}.

\section*{Acknowledgements}
This material is based upon work supported by the National Science Foundation under Grant No. DMS-1439786 while the authors were in residence at the Institute for Computational and Experimental Research in Mathematics in Providence, RI, during the Spring 2017 semester.

The research of RMC is supported in part by the NSF through DMS-1613375 and the Simons Foundation under Grant 354996.  The research of SW is supported in part by the National Science Foundation through DMS-1514910.  The research of MHW is supported in part by the NSF through DMS-1400926.

The authors are grateful to Hongjie Dong for suggestions that substantially improved the results. We also thank Shu-Ming Sun for several helpful conversations.

\appendix
\section{Asymptotics of \texorpdfstring{$V$}{V}} \label{V appendix}

In this appendix, we provide the proof of the asymptotics for $V$.  We begin with the following elementary lemma.
\begin{lemma}\label{lem_velocity}
Let $U\subset \R^n$ and $\omega \in L^1(U) \cap L^p(U)$ for some $1< p \leq \infty$. Then for any $0<s < {n(p-1) / p}$, and $1 \leq q < \infty$ such that $1/p + 1/q = 1$, we have
\begin{equation}
  \label{est K}
  \sup_{x \in \mathbb{R}^n} \int_U {|\omega(y)| \over |x - y|^s} \ dy \leq C \|\omega\|_{L^p(U)}^{qs/n} \|\omega\|_{L^1(U)}^{(n-qs)/n},
\end{equation}
where $C = C(n,q,p,s) > 0$.
\end{lemma}
\begin{proof}
For any $r > 0$, we may estimate
\begin{align*}
\int_U {|\omega(y)| \over |x - y|^s} \, dy & = \int_{U \setminus B_r(x) } {|\omega(y)| \over |x - y|^s} \, dy + \int_{U\cap B_r(x)} {|\omega(y)| \over |x - y|^s} \, dy \\
& \leq {\|\omega\|_{L^1(U)} \over r^s} + \left( {\gamma_n \over n-qs} \right)^{1/q} r^{(n - qs)/q} \|\omega\|_{L^p(U)}. 
\end{align*}
Taking $r^{2/q} := \|\omega\|_{L^1(U)} / \|\omega\|_{L^p(U)}$, we obtain \eqref{est K}.
\end{proof}

We now prove the main result of the appendix.

\begin{proof}[Proof of Lemma \ref{V asymptotics lemma}] Denote 
  \begin{equation}
    \label{K and f} K_{\xi^*}(x,y) := f(x,y) - f(x, \xi^*), \qquad f(x,z) := \frac{x -z}{|x-z|^n}.
  \end{equation}
It is easy to compute that 
\begin{equation}\label{Taylor f}
\begin{split}
\partial_{x_i} f(x,z) & = -\partial_{z_i} f(x,z) = \frac{1}{|x-z|^n} e_i - n \frac{x_i - z_i}{|x-z|^{n+2}} (x-z). 
\end{split}
\end{equation}

We first consider the two-dimensional case. Note that $V$ can be written in the form 
\[ V(x) = \frac{1}{\gamma_2} \int_\Omega \omega(y) K_{\xi^*}(x,y)^\perp \, dy.\]
Now we divide the domain of integration into  the regions 
\[ A := \{y: \ |y| \le |x|^{1-\varepsilon} \} \cap \Omega \quad \textrm{and} \quad B := \{ y: \ |y| > |x|^{1-\varepsilon} \} \cap \Omega.\]

Using \eqref{Taylor f} we compute $K_{\xi^*}$ on $A$ as
\begin{align}
K_{\xi^*}(x,y) = -\frac{y-\xi^*}{|x|^2} + 2 \frac{(y-\xi^*) \cdot x}{|x|^4} x + O\left( \frac{|y - \xi^*|^2}{|x|^3} \right), \quad \text{as } |x| \to \infty. \label{leading order K 2-d}
\end{align}
It follows that 
\begin{align*}
 \int_A \omega(y) K_{\xi^*}(x,y)^\perp \, dy
& = \left( -{m\over |x|^2} + 2{(x\cdot m) \over |x|^{4}} x \right) + O\left( {1\over |x|^{3}} \right), \qquad \textrm{as } |x| \to \infty,
\end{align*}
where we have used the definition of $m$ in \eqref{vortex impulse} and the fact that
\begin{align*}
\left| \int_B \omega(y) (y-\xi^*)  \, dy \right| & \leq \int_B |y|^k |\omega(y)| { |y-\xi^*| \over |y|^k }  \, dy  = O \left( {1\over |x|^{(k-1)(1-\varepsilon)}} \right), \text{ as } |x| \to \infty,
\end{align*}
which follows from moment condition in \eqref{vorticity assumptions}.

On $B$, we apply Lemma \ref{lem_velocity} to estimate 
\begin{align*}
\left| \int_B \omega(y) K_{\xi^*}(x,y)^\perp \, dy \right| & \lesssim \int_B {|\omega(y)| \over |x - y|} \, dy + \int_B {|\omega(y)| \over |x - \xi^*|} \, dy \\
& \lesssim \| \omega \|_{L^1(B)}^{1/2} + {1\over |x|^{k(1-\varepsilon)} |x - \xi^*|} \int_{B} |y|^k|\omega(y)| \, dy \lesssim {1\over |x|^{2+\varepsilon}}.
\end{align*}
Putting together the above computation we obtain the asymptotics \eqref{V asymptotics} for $n=2$.

Now consider the three-dimensional case.  Using the notation introduced in \eqref{K and f}, we may write
\begin{equation}
  \label{3d V}
  V(x) = {1\over \gamma_3} \int_\Omega \omega(y) \times \left[ f(x, 0) + K_0(x, y) \right]\, dy. 
\end{equation}
Notice that the first term involves the total vorticity and has the form ${1\over \gamma_3} \left( \int_\Omega \omega(y)\, dy \right) \times {x\over |x|^3}$.
An application of the divergence theorem leads to the identity
\begin{align}
\int_{\partial (\Omega \cap B_R(0))} y_i \omega \cdot N \, dS = \int_{\Omega \cap B_R(0)} \nabla \cdot (y_i \omega)\, dy & =  \int_{\Omega \cap B_R(0)} \omega_i \, dy. \label{total vorticity 0}
\end{align}
Here additional care is needed due to the low regularity of $\omega$. In particular, the boundary integral is understood as an $H^{1/2}$-$H^{-1/2}$ duality pair, and the last equality holds because $\nabla \cdot \omega = 0$ in the sense of distributions.

From the finite moment assumption \eqref{localized vorticity assumptions}, we know that
$|x|^9\omega \in L^1(\Omega)$. 
Therefore, there exists a sequence of radii $R_j \nearrow +\infty$ such that 
$$\lim_{j\to \infty} \int_{\Omega \cap \partial B_{R_j}(0)} |y|^9 |\omega(y)| \, dy = 0.$$ 
Evaluating \eqref{total vorticity 0} with $R = R_j$, and recalling \eqref{3d vanishing assumption}, we have therefore proved that
\begin{equation}
  \label{net vorticity zero}
  \int_\Omega \omega \, dx = 0. 
\end{equation}
It is quite well-known that the total vorticity is $0$ for three-dimensional solitary waves with compactly supported vorticity; the above argument shows that this remains the case in the more general setting of our localization assumptions \eqref{localized vorticity assumptions}.  

The expansion of the second term in \eqref{3d V} can be treated similarly as in the two-dimensional case. We partition $\Omega$ into the regions $A$ and $B$ defined as before. From \eqref{Taylor f} we have on $A$ that
\begin{align*}
K_0(x, y) = -{y\over |x|^3} + 3{x\cdot y\over |x|^5} x + O\left( {|y|^2\over |x|^4} \right), \qquad \textrm{as $|x|\to \infty$.}
\end{align*} 
Thus
\begin{align*}
{1\over \gamma_3} \int_A \omega(y) \times K_0(x, y) \, dy = {1\over \gamma_3} \int_\Omega \omega(y) \times \left[ -{y\over |x|^3} + 3{x\cdot y\over |x|^5} x \right] \, dy + O\left( {1\over |x|^4} \right). 
\end{align*}
On the other hand, for the integral over  $B$ we estimate 
\begin{align*}
\left| \int_B \omega(y) \times K_0(x, y) \, dy \right| & \le \int_B {|\omega(y)| \over |x - y|^2} \, dy + {1\over |x|^2} \int_{B}|\omega(y)|\, dy \\
& \lesssim \| \omega \|_{L^1(B)}^{1/3}+ {1\over |x|^{k+1-k\varepsilon}} \lesssim {1\over |x|^{3+\varepsilon}},
\end{align*}
where the last two inequalities follow from Lemma \ref{lem_velocity} and assumption \eqref{localized vorticity assumptions}.  Together, these two computations give the asymptotics 
\begin{equation*}
V(x) = {1\over 4\pi} \int_\Omega \omega(y) \times \left[ -{y\over |x|^3} + 3{x\cdot y\over |x|^5} x \right]\, dy + O\left( {1\over |x|^{3+\varepsilon}} \right).
\end{equation*}

The final step is to show that the integral above at leading order involves the vortex impulse.  
For a fixed $1 \leq i, j \leq 3$, the vector field $y_i y_j \omega(y)$ is tangential to $S$ and in $L^1(\Omega)$.  It follows that its divergence (in the distributional sense) must satisfy 
\[
\int_{\Omega} (y_j\omega_i + y_i\omega_j) \, dy = 0.
\]
Multiplying by $x_i$ we find
\begin{align}
  \label{eqn:ing2}
  0 = \int_\Omega \Big(y(\omega(y) \cdot x) + (y\cdot x) \omega(y)
  \Big)\, dy.
\end{align}
Now we rewrite
\begin{align}
  \label{eqn:ing1}
  \omega(y) (x\cdot y) = (\omega(y) \cdot x)y - (\omega(y) \times y) \times
  x.
\end{align}
Integrating \eqref{eqn:ing1} and plugging in \eqref{eqn:ing2}, we have
\begin{align*}
  \int_\Omega \omega(y) (x\cdot y)\, dy &  = -\int_\Omega \omega(y) (x\cdot y)\, dy - \int_\Omega (\omega(y)
  \times y) \times x\, dy \\
  & = - \frac 12 \left(
  \int_\Omega (\omega(y) \times y)\, dy \right) \times x.
\end{align*}
Recalling the definition of the vortex impulse $m$ \eqref{vortex impulse}, this leads to the formula
\begin{align*}
V(x) & = {2m\over \gamma_3 |x|^3} + {3(m\times x)\times x \over \gamma_3 |x|^5} + O\left( {1\over |x|^{3+\varepsilon}} \right)\\
& = {1\over \gamma_3}\left( -{m\over |x|^3} + 3{(x\cdot m) \over |x|^5} x \right) + O\left( {1\over |x|^{3+\varepsilon}} \right) = -{1\over \gamma_3} (m\cdot \nabla) {x\over |x|^3} + O\left( {1\over |x|^{3+\varepsilon}} \right).  \qedhere
\end{align*}
\end{proof}

\section{H\"older regularity under inversion} \label{inversion appendix}

In this appendix we provide a simple lemma that translates decay to H\"older continuity for the spherical inversion.
\begin{lemma}\label{lem inversion}
Let $0 < \alpha < \beta < 1$ be given and set $k := 2\alpha \beta/(\beta - \alpha)$.  Then, if $f \in C_\bdd^\beta(\mathbb R^n)$ satisfies $f(x) = O(1/|x|^k)$ as $|x| \to \infty$, the function $g(x) := f(\tilde x)$ has a $C^\alpha$ extension to a neighborhood of $0$. 
\end{lemma}
\begin{proof}
Let $\alpha$, $\beta$, and $k$ be given as above, and consider the H\"older quotient
\[
{|g(x) - g(y)| \over |x-y|^\alpha},
\] 
for $x$ and $y$ in a neighborhood of the origin.  By symmetry, we can always assume that $|y| \le |x|$. 

Put $\theta := k/\alpha = 2\beta /(\beta - \alpha)$, and suppose first that $|x - y| \ge |x|^\theta$. Then
\begin{align*}
{|g(x) - g(y)| \over |x-y|^\alpha} \lesssim {|x|^k \over |x|^{\theta\alpha}} = 1.
\end{align*}
On the other hand, if $|x-y| \le |x|^\theta$, then $|y| \gtrsim |x|$ as $\theta > 1$.  We may therefore estimate
\begin{align*}
  {|g(x) - g(y)| \over |x-y|^\alpha} 
  &\lesssim {\left|\tilde x - \tilde y \right|^\beta \over |x-y|^\alpha} 
   = {|x - y|^{\beta - \alpha} \over |xy|^\beta}    \lesssim {|x - y|^{\beta - \alpha} \over |x|^{2\beta}} \lesssim {|x|^{(\beta - \alpha) \theta} \over |x|^{2\beta}} = 1. \qedhere \end{align*}
\end{proof}

\bibliographystyle{siam}
\bibliography{projectdescription}

\end{document}